\documentclass[10pt]{amsart}
\usepackage{amsmath,amsthm,amsfonts,amssymb,amscd,flafter,pinlabel}
\usepackage[mathscr]{eucal}
\usepackage{graphics}
\usepackage{caption, subcaption}
\usepackage[usenames,dvipsnames]{color}
\usepackage{graphicx}

\usepackage{pgf}
\usepackage{tikz}
\usetikzlibrary{arrows,automata}
\usepackage[latin1]{inputenc}

\usepackage[colorlinks=true, urlcolor=NavyBlue, linkcolor=NavyBlue, citecolor=NavyBlue, pdftitle={Planar Legendrian graphs}, pdfauthor={Peter Lambert-Cole}, pdfsubject={Contact Topology}, pdfkeywords={Contact Topology, 57M27, 57R58}]{hyperref}

\newtheorem{theorem}{Theorem}[section]
\newtheorem{corollary}[theorem]{Corollary}

\newtheorem{question}[theorem]{Question}
\newtheorem{proposition}[theorem]{Proposition}

\newtheorem{lemma}[theorem]{Lemma}

\newtheorem{innercustomthm}{Theorem}
\newenvironment{customthm}[1]
  {\renewcommand\theinnercustomthm{#1}\innercustomthm}
  {\endinnercustomthm}

\theoremstyle{definition}

\theoremstyle{definition}
\newtheorem{remark}[theorem]{Remark}

\newcommand\cF{\mathcal{F}}

\newcommand\cP{\mathcal{P}}

\newcommand\ZZ{\mathbb{Z}}
\newcommand\RR{\mathbb{R}}

\newcommand\del{\partial}

\newcommand\tb{\text{tb}}

\newcommand\tw{\text{tw}}
\newcommand\rot{\text{rot}}

\begin{document}

\title[{Planar Legendrian graphs}]{Planar Legendrian graphs}

\dedicatory{The first author dedicates this work to Joan Lambert}

\author[P. Lambert-Cole]{Peter Lambert-Cole}
\address{Department of Mathematics \\ Indiana University}
\email{pblamber@indiana.edu}
\urladdr{\href{http://pages.iu.edu/~pblamber/}{http://pages.iu.edu/~pblamber/}}

\author[D. O'Donnol]{Danielle O'Donnol$^{\dag}$ }
\address{Department of Mathematics \\ Indiana University}
\email{odonnol@indiana.edu}
\urladdr{\href{http://pages.iu.edu/~odonnol/}{http://pages.iu.edu/~odonnol/}}

\thanks{$^{\dag}$ This work was partially supported by the National Science Foundation grant DMS-160036.}

\maketitle


\begin{abstract}

We prove two results on the classification of trivial Legendrian embeddings $g: G \rightarrow (S^3,\xi_{std})$ of planar graphs.  First, the oriented Legendrian ribbon $R_g$ and rotation invariant $\rot (g)$ are a complete set of invariants.  Second, if $G$ is 3-connected or contains $K_4$ as a minor, then the unique trivial embedding of $G$ is Legendrian simple.

\end{abstract}

\section{Introduction} 
\label{sec:introduction}

Several topological knot types - the unknot, the figure-8 knot, torus knots \cite{Eliashberg_Fraser,Etnyre_Honda,Ding-Geiges}- are known to be {\it Legendrian simple}.  That is, any Legendrian knot $L$ realizing one of these knot types is determined up to Legendrian isotopy by two classical, homotopy-theoretic invariants of Legendrian knots in $S^3$: the Thurston-Bennequin number $\tb(L)$ and the rotation number $\rot(L)$.  For most other topological types - most famously the $5_2$-knot \cite{Chekanov}- this is not true: there exist Legendrian knots that are smoothly isotopic, have the same Thurston-Bennequin and rotation numbers, but are not isotopic as Legendrian knots.

The classical invariants of Legendrian knots generalize to Legendrian graphs.  A natural question is to determine which spatial graphs are Legendrian simple.  A {\it spatial graph} is an embedding of a fixed abstract graph into a $3$-manifold $M$.  Two spatial embeddings $j_1,j_2$ of an abstract graph are {\it isotopic} if there exists an ambient isotopy $f$ of $M$ such that $f \circ j_1 = j_2$.  A spatial graph is {\it topologically trivial} or {\it unknotted} if its image is embedded on a smoothly embedded 2-sphere.  An abstract graph admits an unknotted embedding in $M$ if and only if it is planar and Mason \cite{Mason} proved that this embedding is unique up to ambient isotopy.  A {\it Legendrian graph} is an embedding $g$ of an abstract graph into a contact 3-manifold $(M,\xi)$ such that the image is tangent to the contact structure at each point.  The image of each cycle is a piecewise-smooth Legendrian knot.  Two Legendrian embeddings $g_1,g_2$ of an abstract graph are {\it Legendrian isotopic} if there exists a contact isotopy $h$ of $(M,\xi)$ such that $h \circ g_1 = g_2$. 

Let $g: G \rightarrow (M,\xi)$ be a Legendrian embedding of an abstract graph.  The {\it Legendrian ribbon} $R_g$ of the Legendrian embedding $g$ is a compact, oriented surface with boundary that is unique up to isotopy rel $g(G)$ and invariant under Legendrian isotopy.  The {\it contact framing} $\overline{R}_g$ is the underlying unoriented surface of the Legendrian ribbon and generalizes the contact or Thurston-Bennequin framing of a Legendrian knot.    The surface $\overline{R}_g$ encodes the contact framing of every cycle and therefore the Thurston-Bennequin numbers of all nullhomologous cycles.  For Legendrian graphs in $S^3$, the rotation and Thurston-Bennequin numbers of each cycle are well-defined and determine invariants $\tb_g,\rot_g \in \ZZ^{|C(G)|}$.  Generalizing from knots, we say that an isotopy class of embeddings of an abstract graph $G$ is {\it Legendrian simple} if the pair $(\overline{R}_g,\rot_g)$ is a complete set of invariants of Legendrian embeddings in this isotopy class, up to Legendrian isotopy.  
In this article we focus on topologically trivial embeddings of planar graphs.  
A planar graph $G$ has a unique trivial embedding up to ambient isotopy.  
When we say the abstract graph $G$ is Legendrian simple this means its unique trivial isotopy class is Legendrian simple.

Some planar graphs are known to be Legendrian nonsimple.  Pavelescu and the second author \cite{OP-Theta} showed that the $\Theta$-graph was not Legendrian simple.  They found two Legendrian embeddings of the $\Theta$-graph with the same contact framing and rotation invariant but that were distinguished by their Legendrian ribbons.  

\begin{remark}
In \cite{OP-Theta}, Legendrian simplicity is defined in terms of the pair $(\tb_g,\rot_g)$.  However, the contact framing $\overline{R}_g$ is the more natural generalization of the Thurston-Bennequin invariant for Legendrian graphs in arbitrary contact manifolds.  Furthermore, their examples remain counter-examples to the present definition of Legendrian simplicity.
\end{remark}

In this paper, we prove two results on the simplicity of planar Legendrian graphs.  First, we give two sufficient conditions for a planar graph to be Legendrian simple.  Let $\Delta_2$ be the graph on 3 vertices with 2 edges connecting each pair of vertices.

\begin{theorem}
\label{thrm:simple}
Let $G$ be an abstract planar graph and $g: G \rightarrow S^3$ its unique topologically trivial embedding.
\begin{enumerate}
\item If $G$ contains $K_4$ or $\Delta_2$ as a minor, then the pair $(\overline{R}_g,\rot_g)$ is a complete set of invariants.
\item If $G$ is 3-connected, the pair $(\tb_g,\rot_g)$ is a complete set of invariants.
\end{enumerate}
In both cases, $G$ is Legendrian simple.
\end{theorem}

Secondly, we show that the only obstruction to Legendrian simplicity is the orientation on the Legendrian ribbon $R_g$.

\begin{theorem}
\label{thrm:complete}
Let $G$ be an abstract planar graph.  The pair $(R_g,\rot_g)$ is a complete set of invariants of topologically trivial Legendrian embeddings $g: G \rightarrow S^3$.
\end{theorem}

An immediate corollary of Theorem \ref{thrm:complete} is

\begin{corollary}
\label{cor:at-most-2}
Let $G$ be an abstract planar graph.  Then there are at most 2 distinct Legendrian isotopy classes of topologically trivial Legendrian embeddings of $G$ with a fixed pair $(\overline{R},\rot)$ classical invariants.
\end{corollary}

In contrast to similar statements for Legendrian knots, we do not obtain these results as a consequence of classifying planar Legendrian graphs.  Instead, we use convex surface theory to directly construct a contact isotopy between a pair of Legendrian embeddings with the same rotation invariant and oriented Legendrian ribbon.

\subsection{Discussion}

For a fixed set of invariants of a topological or geometric object, the geography and botany problems are two standard questions.  The {\it geography problem} for classical invariants of Legendrian graphs asks: Given a fixed isotopy class of spatial graphs, which pairs $(\overline{R},\rot)$ are realized as the invariants of a Legendrian representative of the class?  The {\it botany question}, which refines the geography problem, asks: Given a fixed isotopy class of spatial graphs and fixed pair of invariants $(\overline{R},\rot)$, how many Legendrian representatives $g$ in the class have $\overline{R}_g = \overline{R}$ and $\rot_g = \rot$? 

This paper addresses the botany question.  Corollary \ref{cor:at-most-2} bounds but does not solve the botany question for planar Legendrian graphs.  Theorem \ref{thrm:simple} gives two classes of graphs for which the potential obstruction vanishes.  However, it does not give a complete characterization of those graphs which are nonsimple.  

\begin{question}
Characterize the planar graphs whose trivial embedding is not Legendrian simple.
\end{question}

Secondly, for the study of knotted graphs as opposed to the study of knots, there is a important and subtle distinction between embeddings $g: G \rightarrow M$ and their images.  Let $k: S^1 \rightarrow S^3$ be a knot and $\phi: S^1 \rightarrow S^1$ be an automorphism.  The maps $k$ and $k \circ \phi$ have the same image and since $\phi$ is isotopic to the identity, it is straightforward to construct an ambient isotopy $h_{\phi}$, supported in a tubular neighborhood of $k(S^1)$, such that $h_{\phi} \circ k = k \circ \phi$.  However, the analogous fact is not true for spatial graphs.  Given an abstract graph $G$, a spatial embedding $j: G \rightarrow S^3$, and an automorphism $\phi$ of $G$, it is not true in general that there exists an ambient isotopy $h_{\phi}$ satisfying $h_{\phi} \circ j = j \circ \phi$.  Hence in the classification of spatial graphs it is necessary to distinguish between embeddings of $G$ up to ambient isotopy and their images up to ambient isotopy.

\begin{question}
Characterize the difference between Legendrian simple and Legendrian simple up to reparametrization.
\end{question}

\subsection{Organization}

In Section \ref{sec:background}, we review background material on contact geometry, convex surface theory and the classical invariants of Legendrian graphs.  In Section \ref{sec:botany}, we prove Theorem \ref{thrm:complete} and in Section \ref{sec:simple} we prove Theorem \ref{thrm:simple}.

\subsection{Acknowledgements}
We would like to thank Jeff Meier and Kent Orr for the discussion that inspired this work.



\section{Background}
\label{sec:background}

\subsection{Contact geometry}

A (cooriented) {\it contact structure} $(M,\xi)$ on a 3-manifold $M$ is a plane field $\xi = \text{ker}(\alpha)$ where the 1-form $\alpha$ satisfies the nonintegrability condition $\alpha \wedge d \alpha > 0$.  As a result, the contact structure induces an orientation on $M$ and the 2-form $d \alpha$ orients the contact planes.  The basic example is the standard contact structure $\xi = \text{ker}(dz - y dx)$ on $\RR^3$.  This can be extended by the one-point compactification of $\RR^3$ to the standard structure on $S^3$.  An {\it overtwisted disk} is an embedded disk $D$ in a contact manifold $(M,\xi)$ with Legendrian boundary such that the contact planes are tangent to $D$ along its boundary.  A contact structure is {\it overtwisted} if it contains an overtwisted disk; if not, the contact structure is {\it tight}.  The standard contact structure on $S^3$ is tight.

\subsection{Convex surfaces}
\label{sub:convex}

Let $\Sigma \in (M,\xi)$ be a surface, either closed or with Legendrian boundary.  The restriction of $\xi$ to $\Sigma$ determines a singular line field $\lambda = \xi \cap T\Sigma$ on $\Sigma$ which integrates to a singular foliation $\cF$ on $\Sigma$ called the {\it characteristic foliation}.  The singularities of $\cF$ occur when $\xi$ is tangent to $\Sigma$ and are signed according to whether it is a positive or negative tangency.

The surface $\Sigma \in (M,\xi)$ is {\it convex} if there exists a neighborhood $\nu(\Sigma) = \Sigma \times (-\epsilon,\epsilon)$ in which the contact structure is vertically invariant.  If $\Sigma$ is convex, the {\it dividing set} is a multicurve $\Gamma_{\Sigma}$,  transverse to the leaves of $\cF$ and unique up to isotopy, that divides $\Sigma$ into two subsurfaces $\Sigma_{\pm}$ so that all positive (resp. neg) singularities of $\cF$ lie in $\Sigma_+$ (resp. $\Sigma_-$).  

If $\Sigma$ is convex, then the contact geometry in a neighborhood of $\Sigma$ is determined by the isotopy class of the dividing set.  Giroux's Flexibility Theorem states that we can achieve any characteristic foliation divided by the same multicurve by a $C^{\infty}$-small perturbation of $\Sigma$.  Also, if $\Sigma$ is convex, then Giroux's Criterion states that it has a tight neighborhood if and only if either (1) $\Sigma = S^2$ and the dividing set is connected, or (2) $\Sigma \neq S^2$ and the dividing set has no contractible components.

An arc in $\Sigma$ is Legendrian if and only if it is contained in $\cF$.  Let $C$ be a connected Legendrian arc in a surface $\Sigma$.  If $C$ is not closed, assume its endpoints lie in singularities of the characteristic foliation.  Define $\tw(C,\Sigma) \in \frac{1}{2}\ZZ$ to be the twisting number of the contact planes along $C$ relative to the framing determined by $\Sigma$.  If $C$ connects singularities of opposite sign, the twisting is a half-integer.  Otherwise it is a whole integer.

If a Legendrian arc $C$ lies on a surface $\Sigma$ that is not convex, there is $C^0$-small isotopy of $\Sigma$ rel $C$ to a convex surface if and only if $\tw(C,\Sigma) \leq 0$.  On a convex surface with dividing set $\Gamma$, the twisting number satisfies the formula \cite{Kanda}
\begin{align}
\tw(C,\Sigma) = -\frac{1}{2} \# (C \cap \Gamma)
\end{align}

Let $G$ be a Legendrian graph lying on a surface $\Sigma$.  If $\tw(e,\Sigma) \leq 0$ for all edges of $G$, then there is a $C^0$-small perturbation of $\Sigma$, fixing $G$, so that $\Sigma$ is convex.  Conversely, let $G$ be any graph lying on a convex surface.  If $G$ is {\it nonisolating} - meaning the dividing set $\Gamma$ intersects every component of $\Sigma \smallsetminus G$ - then $G$ can be {\it Legendrian realized} \cite{Kanda,Honda1}.  That is, there is a $C^0$-small isotopy $\phi_t$ of $\Sigma$ through convex surfaces, fixing $\Gamma$, such that $\phi_1(G)$ is Legendrian.  If $G$ is a Legendrian graph lying on a convex surface $\Sigma$, we can assume that its vertices lie at elliptic singularities of the characteristic foliation and its edges are leaves of the foliation.  

Suppose that two convex surfaces $\Sigma_1,\Sigma_2$ with dividing sets $\Gamma_1,\Gamma_2$ meet transversely along a Legendrian knot $L$. Then intersections points of $\Gamma_1 \cap L$ and $\Gamma_2 \cap L$ alternate along $L$.  If $L$ lies in the boundary of $\Sigma_1$ and $\Sigma_2$, then the union $\Sigma = \Sigma_1 \cup_L \Sigma_2$ can be smoothed to a convex surface in a neighborhood of $L$.  Choose coordinates near $L$ so that $\Sigma_1,\Sigma_2$ meet at a right angle.  When viewed from the exterior of the right angle, the dividing set $\Gamma$ of $\Sigma$ is obtained by connecting each arc of $\Gamma_1$ to the arc of $\Gamma_2$ that lies to its right.

\subsection{Bypasses}
\label{sub:bypasses}

A positive {\it bypass disk} is a convex (half)-disk in $(M,\xi)$ with no singularities of its characteristic foliation in its interior and the following four singularities, in order, along its boundary: positive elliptic, positive hyperbolic, positive elliptic, negative elliptic.  A negative bypass disk has the same singularity types with opposite signs.  The bypass disk has a single dividing arc separating the positive and negative singularities.

\begin{figure}
\centering
\includegraphics[width=.8\textwidth]{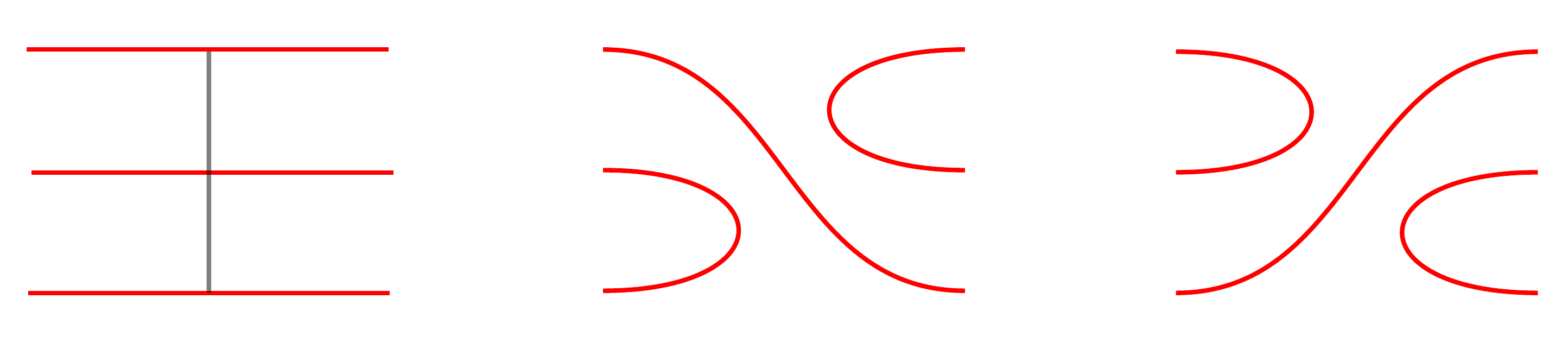}
\caption{Suppose that a bypass can be attached along the arc on the left.  If it is in front, then isotoping over the bypass changes the dividing set as in the middle figure.  If it is behind, the dividing set changes as on the right.}
\label{figure:bypass-effect}
\end{figure}

Let $\Sigma$ be an oriented convex surface with dividing set $\Gamma$.  Let $a$ be an arc in $\Sigma$ that has endpoints in $\Gamma$ and intersects $\Gamma$ once in its interior, as in the first pane of Figure \ref{figure:bypass-effect}.  Suppose that there is a bypass whose intersection with $\Sigma$ is exactly this arc.  Then $a$ is an {\it arc of attachment} for the bypass.  The surface $\Sigma$ can be isotoped over the bypass disk and made convex, but with modified dividing set.  If the bypass lies in front - i.e. on the positively cooriented side of $\Sigma$ - then the dividing set changes as in the middle pane of Figure \ref{figure:bypass-effect}.  If the bypass lies in back - i.e. on the negatively cooriented side - then the dividing set changes according to the right pane.

Bypasses are useful for modifying the isotopy class of the dividing set of a convex surface.  However, some bypasses do not change $\Gamma$ up to isotopy and are called {\it trivial}.  In general, bypasses are hard to find.  However, there are several standard principles that allow us to conclude bypasses exist. 
\begin{enumerate}
\item {\bf Right-to-Life:} Suppose $a$ is an arc in $\Sigma$ corresponding to a trivial bypass attachment.  Then there exists a bypass along $a$.
\item {\bf Imbalance principle:} Suppose that $\Sigma$ is an annulus or disk and has Legendrian boundary $\del \Sigma = c_1 \cup c_2$ with $c_1,c_2$ connected and nonempty with disjoint interiors.   Then if $\tw(c_1,\Sigma) < \tw(c_2,\Sigma)$ then there exists a bypass disk in $\Sigma$ along an arc in $c_2$.
\end{enumerate}

Bypass disks correspond to {\it edge stabilizations} of Legendrian graphs.  Let $e$ be an edge of a Legendrian graph $g: G \rightarrow (S^3,\xi_{std})$ and $B$ a positive bypass disk such that $g(e) \cap B$ is the segment of $\del B$ connecting the positive elliptic singularities.  There is a Legendrian embedding $g_s$ of $G$, smoothly isotopic to $g$ by an isotopy supported in a neighborhood of $B$, such that $g_s(e) \cap B$ is the complementary segment of $\del B$.  If $g(e)$ connects the three positive singularities, the resulting graph $g_s$ is a {\it negative stabilization} of $g$. Conversely, if $g(e)$ connects the three elliptic singularities, the resulting graph is a {\it negative destabilization}.  See Figure \ref{fig:bypass-disk}.  Similarly, if $B$ is negative bypass disk then isotoping an edge across the bypass corresponds to a {\it positive stabilization} or {\it positive destabilization}, as appropriate. 

\begin{figure}
\centering
\labellist
	\small\hair 2pt
	\pinlabel $e$ at 200 25
	\pinlabel $S_-(e)$ at 350 100
	\pinlabel $B$ at 200 70
	\pinlabel $\Sigma_-$ at 100 200
	\pinlabel $\Sigma_+$ at 40 200
	\pinlabel $\Gamma$ at 340 200
\endlabellist
\includegraphics[width=.36\textwidth]{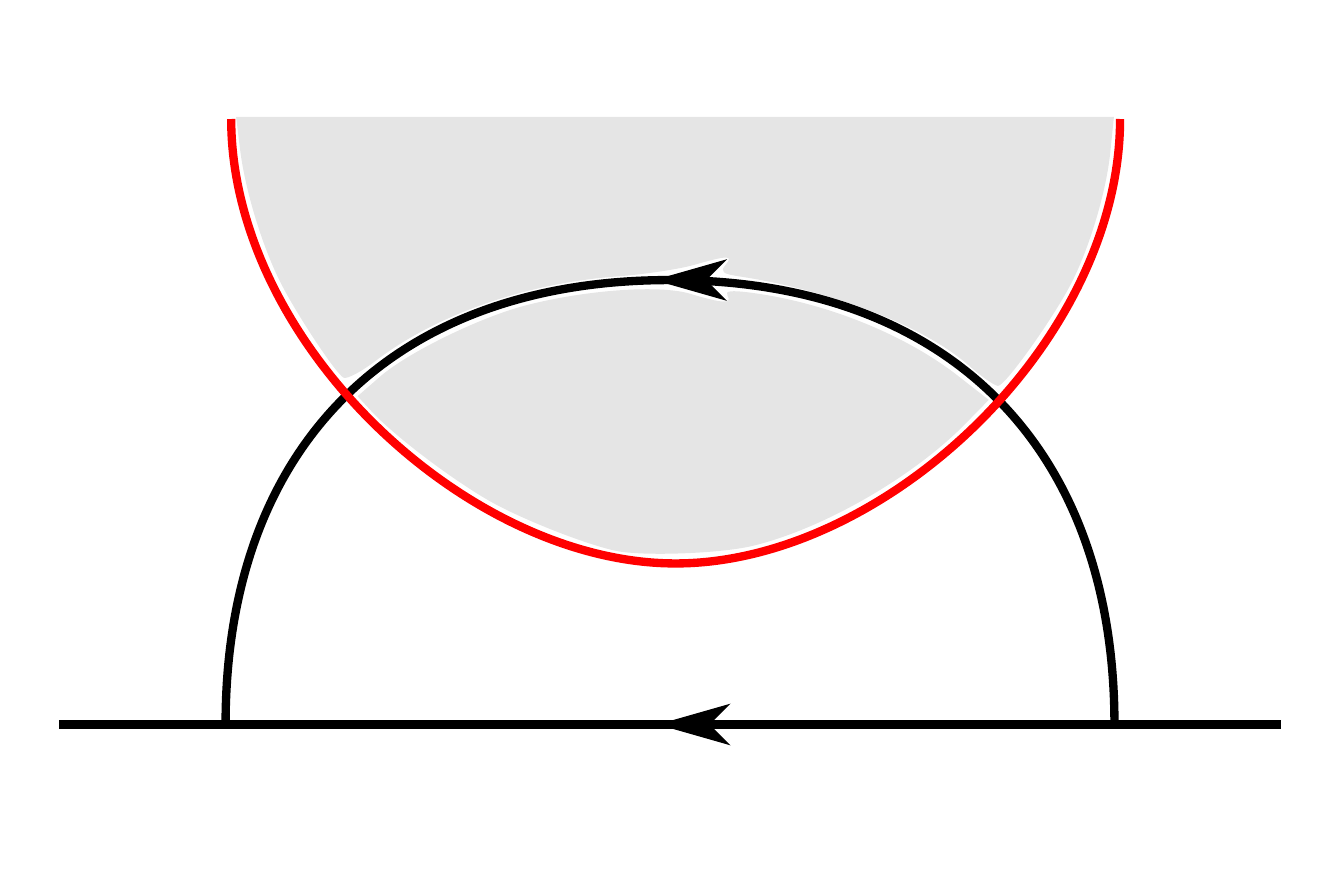}
\caption{Stabilization corresponds to adding a bypass disk to a convex Seifert surface $\Sigma$.  The edge $e$ and its stabilization $S_-(e)$ bound a bypass disk $B$ and a new bigon is formed with the dividing set $\Gamma$.}
\label{fig:bypass-disk}
\end{figure}
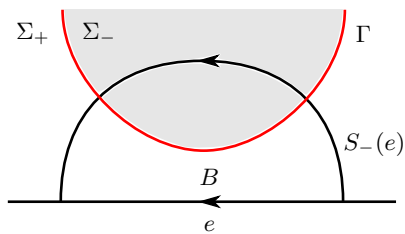

Edge (de)stabilizations can be seen easily on a convex surface.  Let $g$ be a Legendrian graph whose image lies on a convex surface $\Sigma$ with dividing set $\Gamma$. Suppose that $\Gamma$ forms a trivial bigon with some edge $g(e)$.  If removing the bigon does not make $g(G)$ isolating, then we can perturb $\Sigma$ so that it contains a bypass disk along $g(e)$.  Using the bypass disk, we can destabilize $g$.   Conversely,  take any edge $g(e)$ and isotope it to $g_s(e)$ in $\Sigma$ so that $g_s(e)$ and $\Gamma$ form a trivial bigon.  By a perturbation of $\Sigma$, we can ensure that the disk bounded between $g(e)$ and $g_s(e)$ is exactly a bypass disk.  Thus, $g_s$ is a stabilization of $g$.

In summary, a diagrammatic isotopy of $\Gamma$ across edges of $g(G)$ (or vice versa) by introducing or eliminating bigons corresponds to stabilizing and destabilizing $g$.  The signs of the stabilization depends on whether the bigon lies on the positive or negative side of the dividing set.

\subsection{Legendrian ribbon, contact framing and Thurston-Bennequin invariant}
\label{sub:TB}

A {\it framing} or {\it ribbon} for a spatial graph $g: G \rightarrow M$ is a compact surface $F \subset M$ containing $g(G)$ as its 1-skeleton.  If $\gamma$ is a cycle in $G$, then a tubular neighborhood of $g(\gamma)$ in $F$ is a (half-integer) framing for the knot $g(\gamma)$.

A {\it Legendrian ribbon} $R_g$ for a Legendrian graph $g: G \rightarrow (M,\xi)$ is a compact, oriented surface such that
\begin{enumerate}
\item $R_g$ contains $g(G)$ as its 1-skeleton,
\item $\xi$ has no negative tangencies to $R_g$,
\item there exists a vector field $X$ on $R_g$ tangent to the characteristic foliation of $R_g$ whose time-$t$ flow $\phi_t$ satisfies $\cap_{t \geq 0} \phi_t(R_g) = g(G)$,
\item the oriented boundary of $R_g$ is positively transverse to the contact structure $\xi$.
\end{enumerate}
The Legendrian ribbon is unique up to ambient contact isotopy and thus an invariant of $g$.  The underlying unoriented surface $\overline{R}_g$ is the {\it contact framing} of $g$.  If $\gamma$ is a cycle in $G$, there is a contact isotopy, supported near the vertices in $\gamma$, of $g$ to $\widehat{g}$ such that $\widehat{g}(\gamma)$ is a Legendrian knot and a tubular neighborhood of $\widehat{g}(\gamma)$ in $\overline{R}_g$ gives the contact framing of this knot.

At each vertex $v$, the Legendrian ribbon induces an oriented cyclic ordering of the edges incident to $v$.  Along each edge $e$, the embedding $g$ can be parametrized by a path $e(t)$ with $e(0) = v$ and $e'(0) = \lim_{t \rightarrow 0} e'(t)$ a well-defined vector in $\xi_{g(v)}$.  If $v$ has valence $k$, there is a unique indexing of the edges such that $\{e'_1(0),\dots,e'_k(0)\}$ are cyclically ordered in the oriented contact plane $\xi_{g(v)}$.

Let $L$ be a nullhomologous Legendrian knot in $(M,\xi)$.  The {\it Thurston-Bennequin number} $\tb(L)$ is the integral difference between the contact and nullhomologous framings of $L$.  If $L$ is homologically essential, the contact framing is still a well-defined invariant.  If $L_1,L_2$ are topologically-isotopic Legendrian knots, the difference between their contact framings is always a well-defined integer and $L_1,L_2$ have the same contact framing if and only if this integral difference is 0.

If $g: G \rightarrow (M,\xi)$ is a Legendrian graph, then for each cycle $\gamma$ of $G$ the contact framing $\overline{R}_g$ determines the contact framing for the Legendrian knot $g: \gamma \rightarrow (M,\xi)$.  If $g(\gamma)$ is nullhomologous, the cycle $\gamma$ has a well-defined Thurston-Bennequin number $\tb_g(\gamma)$.  For Legendrian graphs in $S^3$ (with any contact structure), all cycles have nullhomologous image.  Given a fixed enumeration of the cycles of $G$, the {\it Thurston-Bennequin cycle invariant} $\tb_g \in \ZZ^{|C(G)|}$ is the vector of Thurston-Bennequin numbers of each cycle \cite{OP-1}.  It is determined by the contact framing $\overline{R}_g$ but the converse is not true in general.

Edge stabilizations change the Thurston-Bennequin invariant.  Suppose that $g_s$ is a stabilization of $g$ along an edge $e$.  Then the Thurston-Bennequin invariants of $g,g_s$ satisfy the relation
\[\tb_{g_s}(\gamma) = \begin{cases}
\tb_g(\gamma) - 1 & \text{if $e \in \gamma$} \\
\tb_g(\gamma) & \text{if $e \notin \gamma$}
\end{cases}\]
Edge destabilization has the opposite effect.

\subsection{Rotation number}
\label{sub:rotation}

The rotation number is the second classical invariant of Legendrian knots.  Using the rotation number, we define two invariants of edge stabilization classes of planar Legendrian graphs.

Let $L$ be an oriented Legendrian knot in $(S^3,\xi_{std})$.  Fix a trivialization of the contact structure $\xi_{std}$ in a 3-ball containing $L$.  The {\it rotation number} $\rot(L)$ is the winding number of $TL \subset \xi_{std}$ in the contact planes relative to this fixed trivialization.  It is independent of the choice of trivialization of $\xi_{std}$.  If $\Sigma$ is a convex Seifert surface for $L$ then the rotation number satisfies \cite{Kanda}
\begin{align}
\rot(L) = \chi(\Sigma_+) - \chi(\Sigma_-)
\end{align}
For a Legendrian graph $g: G \rightarrow (S^3,\xi_{std})$, the rotation number $\rot_g(\gamma)$ of each cycle is well-defined.  Given an enumeration of the cycles of $G$, the {\it rotation invariant} $\rot_g \in \ZZ^{|C(G)|}$ is the vector of rotation numbers of each cycle \cite{OP-1}.

Edge stabilizations change the rotation invariant.  Suppose that $g_s$ is a stabilization of $g$ of sign $\sigma$ along an edge $e$.  Then the rotation invariants of $g,g_s$ satisfy the relation
\[\rot_{g_s}(\gamma) = \begin{cases}
\rot_g(\gamma) + \sigma & \text{if $e \in \gamma$} \\
\rot_g(\gamma) & \text{if $e \notin \gamma$}
\end{cases}\]
Edge destabilization has the opposite effect.

A set of oriented cycles $C = \{\gamma_1,\dots,\gamma_k\}$ in the graph $G$ is {\it fundamental} if (1) each edge is contained in exactly $2k$ cycles, and (2) $k$ cycles containing an edge are oriented in one directions along that edge and the remaining $k$ cycles are oriented in the opposite direction. For a set of oriented cycles $C$, the {\it total rotation number} $\rot_g(C)$ of C is the sum of the rotation numbers of the cycles of $C$:
\[\rot_g(C) := \sum_{\gamma_i \in C} \rot_g(\gamma_i)\]

Let $x,y$ be a 3-connected pair of vertices in $G$.  Recall that this means at least 3 vertices in $G$ must be removed to separate $x$ and $y$. By Menger's Theorem, there exists at least 3 vertex-independent, oriented paths $p_1,p_2,p_3$ from $x$ to $y$.  Let $H$ be the union $p_1 \cup p_2 \cup p_3$, which is a subdivision of a $\Theta$-graph.  If $g: G \rightarrow M$ is an embedding, then $g$ restricts in an obvious way to an embedding $h: H \rightarrow M$ and if $g$ is Legendrian so is $h$.  The graph $H$ has three oriented cycles $C_1 = p_1 - p_2; C_2 = p_2 - p_3; C_3 = p_3 - p_1$ and the set $C = \{C_1,C_2,C_3\}$ is fundamental.  If $g$ is Legendrian embedding of $G$, label the paths such that $p_1,p_2,p_3$ are positively cyclically ordered in the contact plane $\xi_{g(x)}$.  

\begin{lemma}
\label{lemma:total-rot-theta}
The total rotation number of $h$ over the set $C$ satisfies
\[\rot_h(C) = \begin{cases} 1 & \text{if $p_1,p_2,p_3$ are negatively cyclically orded in $\xi_{g(y)}$} \\ 0 & \text{if $p_1,p_2,p_3$ are positively cyclically orded at $\xi_{g(y)}$} \end{cases}\]
\end{lemma}

\begin{proof}
Up to switching the orientations, this statement is an immediate consequence of the proof of Lemma 5 in \cite{OP-Theta}. 
\end{proof}

\subsection{Classifying Legendrian graphs}
\label{sub:classifying}

Convex surface theory contains powerful techniques for classifying Legendrian knots.  The following proposition is a standard tool and its proof is based on deep results in contact topology.

\begin{proposition}
\label{prop:dividing-equals-isotopy}
Let $G_1,G_2$ be Legendrian graphs lying on convex surfaces $S_1,S_2$ in $(S^3, \xi_{std})$.  
\begin{enumerate}
\item Suppose that there is a diffeomorphism $i: S_1 \rightarrow S_2$ that sends $G_1$ diffeomorphically to $G_2$ and the dividing set $\Gamma_{S_1}$ diffeomorphically to the dividing set $\Gamma_{S_2}$.  Then there is a contactomorphism of neighborhoods $j: \nu(S_1) \rightarrow \nu(S_2)$ that sends $G_1$ to $G_2$.
\item If $S_1,S_2 \simeq S^2$ then $j$ extends to a contactomorphism $S^3 \rightarrow S^3$.
\item Suppose that $j$ extends to a contactomorphism $S^3 \rightarrow S^3$.  Then $G_1$ and $G_2$ are Legendrian isotopic in $S^3$.
\end{enumerate}
\end{proposition}

\begin{proof}
Let $\cF$ be the characteristic foliation on $S_1$ and $i(\cF)$ its image on $S_2$.  Then, Giroux's Flexibility Theorem ensures that we can find a $C^0$-small isotopy of $S_2$ to $S'_2$, fixing $\Gamma_{S_2}$, so that its characteristic foliation is exactly $i(\cF)$.  By Giroux's Reconstruction Lemma, the characteristic foliation determines the contact structures in neighborhoods of $S_1,S'_2$.  As a result, the diffeomorphism $i$ extends to a contactomorphism $j$ of neighborhoods.

Secondly, all smooth spheres in $S^3$ separate it into two 3-balls.   The standard contact structure on $S^3$ is tight \cite{Bennequin} and is unique \cite{Eliashberg92}.  In addition, if $\xi_1,\xi_2$ are tight contact structures on the 3-ball inducing the same characteristic foliation on the boundary, then they are isotopic \cite{Eliashberg92}.  As a result, $j$ extends to the entirety of $S^3$.

Finally, fix a point $p \in S^3$, let $\xi_p = \xi_{std}(p)$ be the hyperplane of the standard contact structure at $p$.  Define $\text{Diff}_0(S^3)$ to be the set of diffeomorphisms of $S^3$ that fix $\xi_p$ and $\text{Diff}_{\xi_{std}}$ the group of diffeomorphisms that fix $\xi_{std}$. Eliashberg \cite{Eliashberg92} proved that the map
\[\text{Diff}_{\xi_{std}} \hookrightarrow \text{Diff}_0(S^3)\]
is a weak homotopy equivalence.  Thus, we can replace $j$ with a contact isotopy that sends $G_1$ to $G_2$.
\end{proof}

\section{Botany problem}
\label{sec:botany}
In this section, we prove Theorem \ref{thrm:complete} and show that the pair $R_g,\rot_g$ is a complete set of invariants for topologically trivial Legendrian graphs.  First, in Subsection \ref{sub:planar} we prove some preliminary results on topologically trivial Legendrian graphs and fix conventions for the proof.  We split the main proof into two parts: in Subsection \ref{sub:base} we prove it for the subclass $\cP_0$ of Legendrian embeddings that can be realized on a convex sphere, then in Subsection \ref{sub:general} we extend it to the general case.

\subsection{Planar Legendrian graphs}
\label{sub:planar}

Let $G$ be an abstract, connected planar graph and $g: G \rightarrow (S^3,\xi_{std})$ a trivial Legendrian embedding.  
In addition, all 2-valent vertices of $G$ can be ignored, since they can simultaneously be smoothed through ambient isotopy.  
Since $g$ is trivial, there exists a smoothly embedded, oriented 2-sphere $\Sigma \subset S^3$ containing the image of $g$.  After a perturbation of $\Sigma$, we can assume that $\xi$ has an isolated tangency to $\Sigma$ at each vertex $g(v)$.  The contact plane $\xi_{g(v)}$ is oriented and let $\sigma(v,\Sigma)$ be the sign of this tangency.  For each edge $e$ of $G$, let $\tw(e,\Sigma) \in \frac{1}{2} \ZZ$ be the twisting of $\xi$ relative to $\Sigma$ along $e$.  Equivalently, $\tw(e,\Sigma)$ is the relative difference along $e$ of the framings induced by $R_g$ and $\Sigma_g$.  If the endpoints of $e$ have the same signs, then $\tw(e,\Sigma)$ is a whole integer, and if they have opposite signs, then $\tw(e,\Sigma)$ is a half-integer.  If $\gamma$ is a cycle, then
\[\tb(\gamma) = \sum_{e \in \gamma} \tw(e,\Sigma)\]
since the surface framing in $\Sigma$ is exactly the nullhomologous framing, as the cycle bounds a disk in $\Sigma$.

Let $g: G \rightarrow (S^3,\xi_{std})$ be a Legendrian graph and let $\Sigma$ be an oriented, embedded sphere containing $g(G)$.  Let $\gamma_1,\gamma_2$ be two cycles of $G$ whose images in $\Sigma$ bound disks $D_1,D_2$ with disjoint interiors.  Orient $\gamma_1,\gamma_2$ as the boundaries of $D_1,D_2$.  Define the {\it boundary connect sum} $\gamma_1 \natural \gamma_2$ of $\gamma_1, \gamma_2$ to be the oriented resolution of $g(\gamma_1) \cup g(\gamma_2)$ in $\Sigma$ at $v$ into a connected simple closed curve.  By a perturbation of $\Sigma$, we can assume that $\gamma_1 \natural \gamma_2$ is a Legendrian knot.

\begin{lemma}
\label{lemma:del-connect-sum}
Let $\gamma_1,\gamma_2$ be cycles that intersect at $v$.  Then
\begin{align*}
\tb(\gamma_1 \natural \gamma_2) &= \tb_g(\gamma_1) + \tb_g(\gamma_2) \\
\rot(\gamma_1 \natural \gamma_2) &= \rot_g(\gamma_1) + \rot_g(\gamma_2) - \sigma(v,\Sigma)
\end{align*}
\end{lemma}

\begin{proof}
Since $\gamma_1 \natural \gamma_2$ is a resolution of $\gamma_1 \cup \gamma_2$ at $g(v)$, the intersection numbers satisfy
\[
-\frac{1}{2} \# \Gamma \cap (\gamma_1 \natural \gamma_2)  =  -\frac{1}{2} ( \# \Gamma \cap \gamma_1 + \# \Gamma \cap \gamma_2) \]
Thus, the Thurston-Bennequin number is additive.  Secondly, let $D$ be the disk bounded by $\gamma_1 \natural \gamma_2$.  If $\sigma(v,\Sigma) = +1$, then
\begin{align*}
\chi(D^+) &= \chi(D^+_1) + \chi(D^+_2) - 1 \\
\chi(D^-) &= \chi(D^-_1) + \chi(D^-_2)
\end{align*}
since the resolution introduces an extra 1-handle connecting $D^+_1$ to $D^+_2$.  As a result, 
\begin{align*}
\rot(\gamma_1 \natural \gamma_2) &= \chi(D^+) - \chi(D^-) \\
&= \chi(D^+_1) + \chi(D^+_2) - 1 - \left(\chi(D^-_1) + \chi(D^-_2)\right) \\
&= \chi(D^+_1) - \chi(D^-_1) + \chi(D^+_2) - \chi(D^-_2) - 1\\
&= \rot_g(\gamma_1) + \rot_g(\gamma_2) - 1
\end{align*}
A similar argument proves the lemma if $\sigma(v,\Sigma) = -1$.
\end{proof}

\subsubsection{The surface $\Sigma_g$}
An edge $e$ of $G$ is a {\it cut edge} if $G \smallsetminus e$ is disconnected and a pair of edges $e_1,e_2$ are a {\it cut pair of edges} if $G \smallsetminus ( e_1 \cup e_1)$ is disconnected.  Fix an embedding $g: G \rightarrow S^2$.  An edge $e$ is a cut edge if and only if the same face of $S^2 \smallsetminus g(G)$ lies on both sides of $g(e)$.  Similarly, a pair $e_1,e_2$ are a cut pair of edges if and only if they both lie in the boundaries of an adjacent pair of faces.
A vertex $v$ of $G$ is a {\it cut vertex} if deleting $v$ (and all incident edges) results in a disconnected graph. 
A pair of vertices $v_1, v_2$ of $G$ are a {\it cut pair of vertices} if deleting both of the vertices (and all incident edges) results in a disconnected graph. 

\begin{lemma}
\label{lemma:half-twist-sphere}
Let $g: G \rightarrow (M,\xi)$ be a trivial Legendrian embedding and $\Sigma$ a 2-sphere containing $g(G)$.  Fix $N \in \frac{1}{2} \ZZ$.
\begin{enumerate}
\item If $e$ is a cut edge of $G$, then there exists a 2-sphere $\Sigma'$ containing $g(G)$ such that 
\[\tw(e,\Sigma') = N\]
\item If $e_1,e_2$ are a cut pair of edges of $G$, then there exists a 2-sphere $\Sigma'$ containing $g(G)$ such that 
\[\tw(e_1,\Sigma') = N \quad \text{and} \quad \tw(e_2,\Sigma') = \tw(e_1,\Sigma) + \tw(e_2,\Sigma) - N\]
\end{enumerate}
Furthermore, in both cases $\tw(e',\Sigma) = \tw(e',\Sigma')$ for all other edges $e'$ of $G$.
\end{lemma}

\begin{proof}
First, let $e$ be a cut edge.  Then there is a simple closed curve $\delta$ in $\Sigma$ that intersects $G$ in exactly one point along $e$.  Let $D$ be the disk bounded by this curve.  Choose a coordinate chart on a neighborhood $U$ of $D$ in $S^3$ that sends $D$ to the unit disk in the $xy$-plane and the segment of $e$ outside $D$ to the positive $x$-axis.  Let $h_{\frac{1}{2}}$ be an isotopy, supported in $U$, that rotates the unit disk around the $x$-axis by the angle $-\pi$.  Define $h_{k} = (h_{\frac{1}{2}})^{2k}$ for $k \in \frac{1}{2}\ZZ$.  The surface $h_{N - \tw(e,\Sigma)}(\Sigma)$ is the required surface.

Secondly, if $e_1,e_2$ are a pair of cut edges then there exists a simple closed curve $\delta$ in $\Sigma$ that intersects $G$ in exactly two points, once in $e_1$ and once in $e_2$, bounding a disk $D$. Define a similar isotopy $h_{\frac{1}{2}}$ of $g(G)$ in $S^3$ that increments $\tw(e_1,\Sigma)$ by $\frac{1}{2}$ and $\tw(e_2,\Sigma)$ by $-\frac{1}{2}$.  Define $h_k$ as above.  As a result, $h_{N - \tw(e_1,\Sigma)}(\Sigma)$ is the required surface.
\end{proof}

\newpage
\begin{corollary}
\label{cor:sigma-g}
Let $g$ be a Legendrian embedding of $G$.  There exists a smoothly embedded sphere $\Sigma_g$ containing $g(G)$ such that
\begin{enumerate}
\item if $e$ is a cut edge, then $\tw(e,\Sigma_g) = 0$, and
\item if $e_1,\dots,e_k$ are the edges in the common boundary of two adjacent faces of $\Sigma_g \smallsetminus g(G)$, then 
\[\tw(e_2,\Sigma_g) = \cdots = \tw(e_k,\Sigma_g) = 0\]
\end{enumerate}
\end{corollary}

From this point forward, let $\Sigma_g$ denote a surface satisfying the conclusions of Corollary \ref{cor:sigma-g}.

\begin{lemma}
\label{lemma:ribbon-embedding}
Let $G$ be a planar graph and let $g_1,g_2$ be trivial Legendrian embeddings.  If $g_1,g_2$ have the same Legendrian ribbon, then the  pairs $(\Sigma_{g_1},g_1(G))$ and $(\Sigma_{g_2},g_2(G))$ are diffeomorphic. 
\end{lemma}

\begin{proof}
By Whitney's Theorem \cite{Whitney}, if $G$ is 3-connected then all embeddings $g: G \rightarrow S^2$ are equivalent up to homeomorphism of $S^2$.  

If $G$ is not 3-connected, we can use $R_g$ to define an extension $H$ that is 3-connected as follows.  Recall that if $e$ is a cut edge then $\tw(e,\Sigma_g) = 0$.  Since there is no twisting of $R_g$ relative to $\Sigma_g$ along $e$, isotope $R_g$ to lie in $\Sigma_g$ along $e$.  Add two extra edges parallel to $e$, one on either side.  This extension is still planar since it lies in $\Sigma_g$ but it also clearly only depends on $R_g$. If $e_1,e_2$ form a cut pair, then up to relabeling we can assume that $\tw(e_2,\Sigma_g) = 0$.  Similarly, add a pair of edges parallel to $e_2$.  Once this is accomplished for all cut edges and cut pairs of edges, the resulting graph $H'$ is 3-edge-connected.

Now, let $v$ be a vertex of $H'$ with incident edges $e_1,\dots,e_k$ in oriented cyclic order.  Near $v$, there is an isotopy of $R_g$ so that it lies in $\Sigma_g$ and therefore $R_g$ determines the cyclic ordering at $v$.  Subdivide each edge to introduce $k$ new vertices $v_1,\dots,v_k$.  Now attach edges connecting $v_i$ to $v_{i+1}$ for $i=1,\dots,k$.  Let $H$ be the graph obtained by this procedure at every vertex of $H'$.  
Since $H'$ is 3-edge connected any cut vertex $v$ of $H'$ will have at least three edges to each of the components that result from its deletion.  
So in $H$ it will take at least three vertex deletions to disconnect such a component.  
Similarly, pairs of cut vertices of $H'$ are no longer pairs of cut vertices in $H$.  
The resulting graph $H$ is 3-connected.  Thus, the embedding of $H \hookrightarrow \Sigma_g$ is unique up to homeomorphism of $S^2$.  Moreover, the construction of $H$ depended only on $R_g$.
\end{proof}

\subsubsection{Overview of the proof of Theorem \ref{thrm:complete}}

Let $p(g)$ be the number of edges of $G$ such that $\tw(e,\Sigma_g) > 0$ and define $\cP_k$ to be the set of Legendrian embeddings $g$ such that $p(g) \leq k$.  This defines an increasing sequence of subclasses
\[\cP_0 \subset \cP_1 \subset \dots \subset \cP_n \subset \dots\subset\cP_{|E(G)|}\]
The surface $\Sigma_g$ can be made convex with $g(G)$ lying in its characteristic foliation if and only if $\tw(e,\Sigma_g) \leq 0$ for all edges $e$ of $G$.  Thus, only Legendrian graphs in $\cP_0$ can be Legendrian realized on a convex 2-sphere.  In general, we can realize graphs in $\cP_k$ on convex surfaces of genus $k$.  

In Subsection \ref{sub:base}, we will first prove Theorem \ref{thrm:complete} for Legendrian embeddings in $\cP_0$ and then in Subsection \ref{sub:general} extend this to arbitrary trivial Legendrian embeddings.

\subsection{Legendrian graphs in $\cP_0$}
\label{sub:base}

First we prove Theorem \ref{thrm:complete} for Legendrian graphs in the subclass $\cP_0$.

If $g \in \cP_0$ then there exists a $C^0$-small perturbation of $\Sigma_g$ fixing $g(G)$ such that $\Sigma_g$ is convex and contains $g(G)$ in its characteristic foliation.  The convex sphere $\Sigma_g$ has a single dividing curve $\Gamma$, since $(S^3,\xi_{std})$ is tight.  The dividing curve separates $\Sigma_g$ into two components $\Sigma^{\pm}_g$ containing the positive and negative tangencies of $\xi$ to $\Sigma$, respectively.  Orient $\Gamma$ as the boundary of the positive region $\Sigma^+_g$ 

Let $F$ be a face of $\Sigma_g \smallsetminus g(G)$.  Since $G$ is connected, $F$ is a topological disk.  Choose a pushoff of $\del F$ into $F$ and perturb $\Sigma_g$ so that this pushoff is Legendrian.   By abuse of notation, we use $\del F$ to denote this Legendrian unknot.  The classical invariants of $\del F$ are $\tb(\del F) = \sum_{e \in \del F} \tw(e,\Sigma_g)$ and $\rot(\del F) = \chi(F^+) - \chi(F^-)$.  

\begin{lemma}
Let $F$ be a face of $\Sigma_g$.  The invariants $\tb(\del F),\rot(\del F)$ are determined by $R_g,\rot_g$.
\end{lemma}

\begin{proof}
Without loss of generality, we can assume that $F$ is not incident to any cut edges.  By Corollary \ref{cor:sigma-g}, we can assume that $\tw(e,\Sigma) = 0$ along each cut edge $e$.  Thus, contracting the edge $e$ does not change the Legendrian isotopy class of $\del F$.  

If $F$ is not incident to any cut edges, then the boundary of $F$ is the union of a collection of cycles $\gamma_1,\dots,\gamma_k$, identified at cut vertices of $G$.  Consequently, $\del F$ is the boundary connect sum $\gamma_1 \natural \dots \natural \gamma_k$.  So by Lemma \ref{lemma:del-connect-sum}, the classical invariants of $\del F$ can be computed from the classical invariants of $\gamma_1,\dots,\gamma_k$ and the signs of vertices, which are determined by $R_g,\rot_g$.
\end{proof}

The dividing set intersects $\del F$ in $2 \cdot \tb(\del F)$ points and intersects each edge $e$ of $\del F$ in $2 \cdot \tw(e,\Sigma_g)$ points.  Starting with a positive intersection point (with respect to the oriented intersection of $\del F$ and $\Gamma$), label the points $\del F \pitchfork \Gamma$ as $x_1,\dots, x_{2 \cdot \tb(\del F)}$.  The signs of the intersections alternate along $\del F$, so the sign of $x_i$ is $-(-1)^i$.  In addition, if a vertex $v$ lies along $\del F$ between $x_i$ and $x_{i+1}$, then $\sigma(v,\Sigma_g) = -(-1)^i$.

Let $g_1,g_2$ be two Legendrian embeddings of $G$ with the same invariants $R_g,\rot_g$.  Let $\Sigma_i = \Sigma_{g_i}$ be the convex sphere containing $g_i$ in its characteristic foliation and let $\Gamma_i$ be the dividing curve of $\Sigma_i$.  By Lemma \ref{lemma:ribbon-embedding}, there exists a diffeomorphism $j: (\Sigma_{1},g_1(G)) \rightarrow (\Sigma_{2},g_2(G))$.  Since $g_1,g_2$  have the same Legendrian ribbon, they have the same edge invariant $\tw(e,\Sigma) = -\frac{1}{2} \# g(e) \cap \Gamma$.  Thus, we can assume $j$ sends $g_1(G) \cap \Gamma_1$ to $g_2(G) \cap \Gamma_2$.  In addition, the Legendrian ribbon determines the sign of tangency of $\xi$ at each vertex.  Thus, if $F_1$ and $F_2$ are corresponding faces,  the map $j$ preserves the orientations of the points of $\del F_i \cap \Gamma_i$.

In order to apply Proposition \ref{prop:dividing-equals-isotopy} and conclude $g_1,g_2$ are ambient isotopic, we need to find sequences of bypasses on $\Sigma_1$ and $\Sigma_2$ so that $j(\Gamma_1) = \Gamma_2$.

The following lemma is our main tool for finding bypasses.

\begin{lemma}
\label{lemma:sphere-bypass}
Every potential arc of attachment $a$ on $S^2$ corresponds to a trivial bypass attachment either in front or in back.
\end{lemma}

\begin{proof}
Since $\xi_{std}$ is tight, Giroux's Criterion implies that the dividing set is connected.  Therefore, up to isotopy, there are two possible attaching arcs.  See Figure \ref{figure:sphere-bypass}.  Attaching a bypass disk in front along the first arc is trivial and attaching a bypass in back along the second arc is also trivial.   The Right-to-Life principle ensures that such bypasses exist.
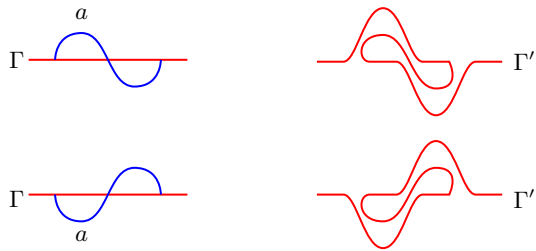
\begin{figure}[htpb!]
\centering
\labellist
	\small\hair 2pt
	\pinlabel $a$ at 120 280
	\pinlabel $a$ at 120 40
	\pinlabel $\Gamma$ at 50 80
	\pinlabel $\Gamma$ at 50 230
	\pinlabel $\Gamma'$ at 600 80
	\pinlabel $\Gamma'$ at 600 230
\endlabellist
\includegraphics[width=.45\textwidth]{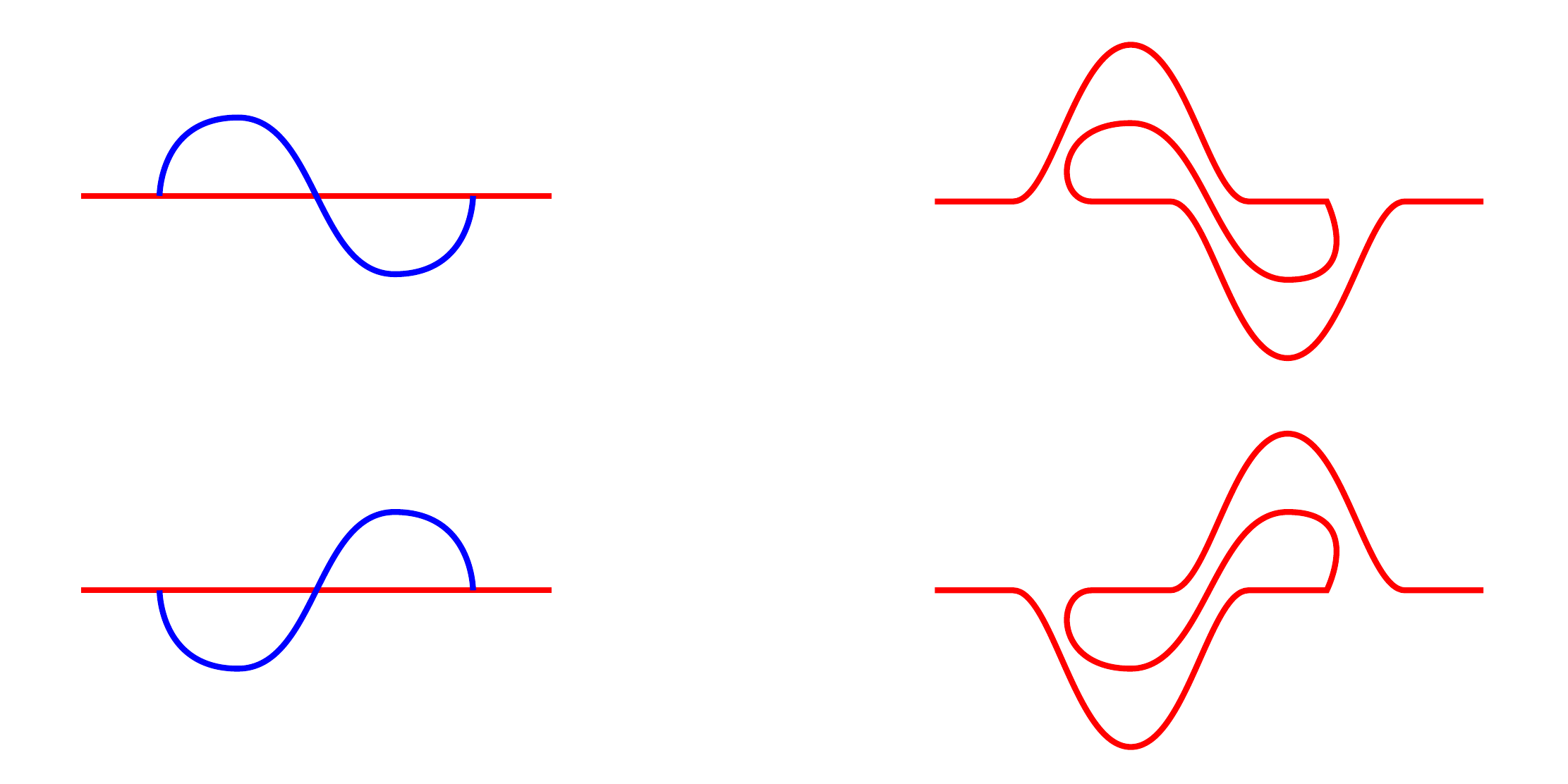}
\caption{Attaching a trivial bypass along the arc $a$ either in front the sphere (top line) or behind the sphere (bottom line)}
\label{figure:sphere-bypass}
\end{figure}
\end{proof}

\begin{proposition}
\label{prop:complete}
Let $g_1,g_2$ be trivial Legendrian embeddings of $G$ with the same invariants $R_g,\rot_g$.  Suppose that $g_i(G)$ lies on the convex sphere $\Sigma_{g_i}$ with dividing curve $\Gamma_i$.  Then 
\begin{enumerate}
\item Let $F_1,F_2$ be corresponding faces.  There exists sequences of bypasses  attached to $\Sigma_{g_1},\Sigma_{g_2}$ in the interiors of $F_1,F_2$ such that, after attaching these bypasses, the dividing sets in $F_1,F_2$ are isotopic rel boundary.
\item There exists sequences of bypasses on $\Sigma_{g_1}$ and $\Sigma_{g_2}$ such that, after attaching these bypasses, the dividing curves $\Gamma_1,\Gamma_2$ are isotopic rel $g_1(G),g_2(G)$.  
\item The Legendrians embeddings $g_1,g_2$ are isotopic.
\end{enumerate}
\end{proposition}

\begin{remark}
Statement (1) of Proposition \ref{prop:complete} implies that the unknot is Legendrian simple.  However, the proof in \cite{Eliashberg_Fraser} is not sufficient for our purposes.  In particular, Eliashberg and Fraser eliminate positive elliptic singularities of the characteristic foliation along the boundary of a Seifert disk.  However, the vertices of a Legendrian graph occur at elliptic singularities of both signs and cannot be eliminated.
\end{remark}

\begin{proof}
Statement (2) follows by applying Statement (1) to each face and Statement (3) follows from Statement (2) by Proposition \ref{prop:dividing-equals-isotopy}.  Thus, we just need to prove Statement (1) of the proposition.

To prove Statement (1), we induct on $|\tb(\del F)|$.  If $\tb(F_1) = \tb(F_2) = -1$, then $\Gamma_i \cap \del F_i$ consists of two points and $\Gamma_i \cap F_i$ is a single arc that is unique up to isotopy.

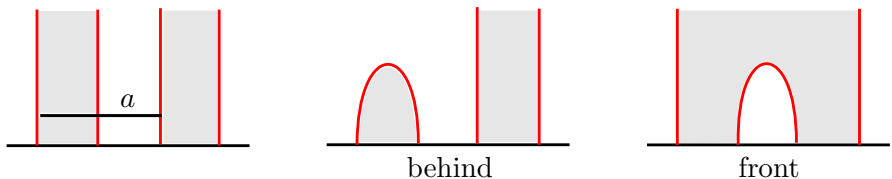
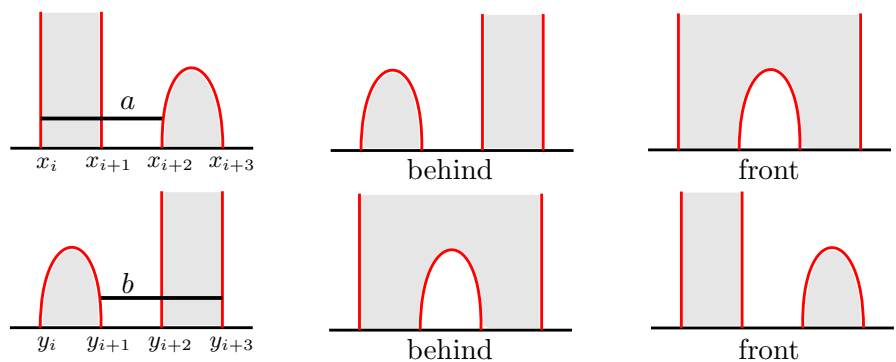
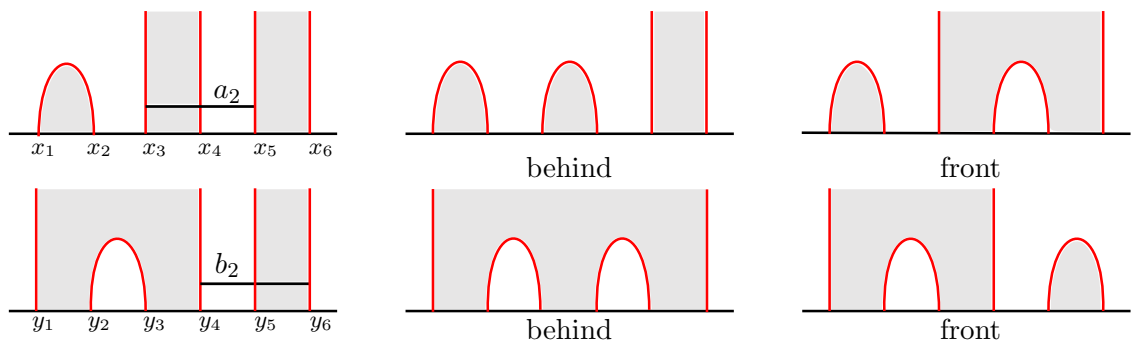
\begin{figure}
\centering
\begin{subfigure}{.8\textwidth}
	\centering
	\labellist
		\large\hair 2pt
		\pinlabel $a$ at 160 80
		\pinlabel behind at 500 10
		\pinlabel front at 837 10
	\endlabellist
	\includegraphics[width=1\textwidth]{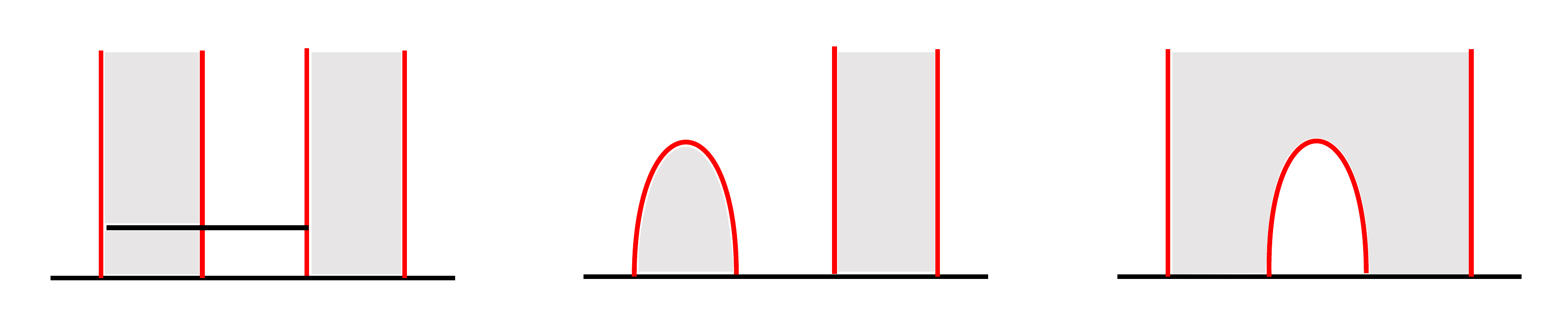}
	\caption{Attaching a trivial bypass along $\del F$ either in front (middle) or behind (right) produces a boundary-parallel dividing arc.}
	\label{subfig:bypass-options}
\end{subfigure}
\begin{subfigure}{.8\textwidth}
	\centering
	\labellist
		\large\hair 2pt
		\pinlabel $b$ at 160 80
		\pinlabel $a$ at 160 270
		\pinlabel behind at 500 10
		\pinlabel front at 837 10
		\pinlabel behind at 500 200
		\pinlabel front at 837 200
		\small\hair 2pt
		\pinlabel $y_{i}$ at 75 15
		\pinlabel $y_{i+1}$ at 140 15
		\pinlabel $y_{i+2}$ at 205 15
		\pinlabel $y_{i+3}$ at 270 15
		\pinlabel $x_{i}$ at 75 205
		\pinlabel $x_{i+1}$ at 140 205
		\pinlabel $x_{i+2}$ at 205 205
		\pinlabel $x_{i+3}$ at 270 205
	\endlabellist
	\includegraphics[width=1\textwidth]{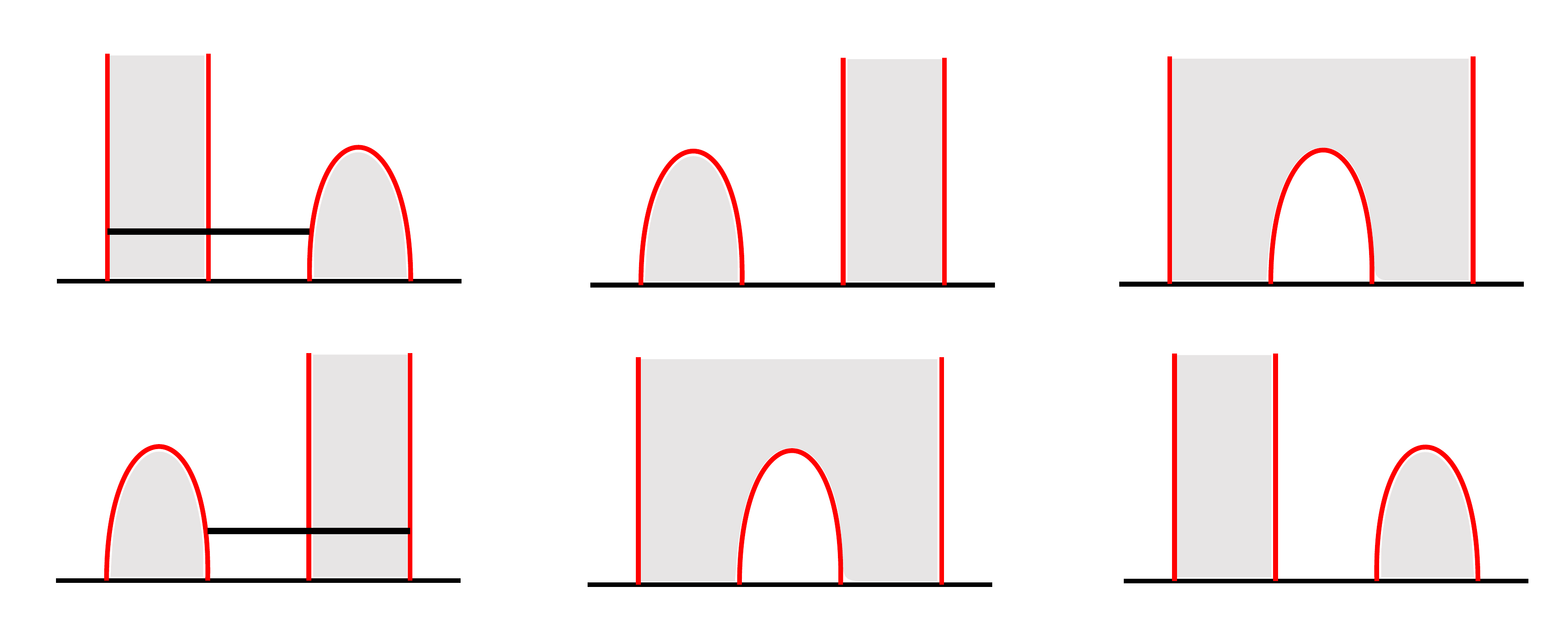}
	\caption{{\bf Adjacent bypasses:} If $\del F_1,\del F_2$ have adjacent bypasses, then after a sequence of bypass attachments the faces $F_1,F_2$ have matched bypasses.}
	\label{subfig:adjacent-options}
\end{subfigure}
\begin{subfigure}{1.0\textwidth}
	\centering
	\labellist
		\large\hair 2pt
		\pinlabel $b_2$ at 290 85
		\pinlabel $a_2$ at 290 285
		\pinlabel behind at 690 10
		\pinlabel front at 1160 10
		\pinlabel behind at 690 200
		\pinlabel front at 1160 200
		\small\hair 2pt
		\pinlabel $y_1$ at 75 15
		\pinlabel $y_2$ at 140 15
		\pinlabel $y_3$ at 205 15
		\pinlabel $y_4$ at 270 15
		\pinlabel $y_5$ at 335 15
		\pinlabel $y_6$ at 400 15
		\pinlabel $x_1$ at 75 220
		\pinlabel $x_2$ at 140 220
		\pinlabel $x_3$ at 205 220
		\pinlabel $x_4$ at 270 220
		\pinlabel $x_5$ at 335 220
		\pinlabel $x_6$ at 400 220
	\endlabellist
	\includegraphics[width=1\textwidth]{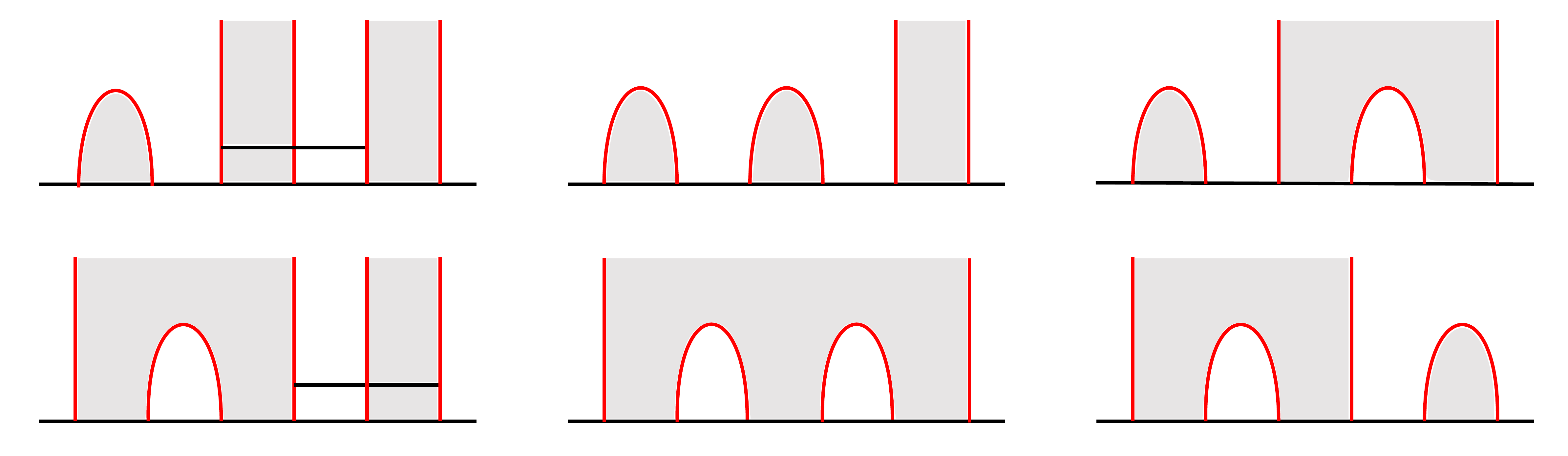}
	\caption{At the $2^{\text{nd}}$ step, attaching the bypasses along $a_2$ and $b_2$ produces either (1) matched bypasses, (2) adjacent bypasses, or (3) increases the sequence of positive bypasses in $F_1$ and negative bypasses in $F_2$.}
	\label{subfig:induction-bypass-options}
\end{subfigure}
\caption{The possible effects on the dividing sets of attaching bypasses near the boundary of $F_1,F_2$.  In each figure, the effect of attaching behind is in the middle and attaching in front is on the right.}
\end{figure}

Now suppose that $\tb(\del F_1) = \tb (\del F_2) = -n$ and the statement is true for faces $F$ with $\tb(F) = -n + 1$.  We prove the inductive step in 3 parts: (A)  if $F_1,F_2$ have {\it matched bypasses}, then we can reduce to the case when $\tb(\del F_1) = \tb(\del F_2) = -n +1$; (B) if $F_1,F_2$ have {\it adjacent bypasses}, then we can modify the dividing sets to find matched bypasses; and finally (C) we can always ensure that matched bypasses or adjacent bypasses occur.

{\bf Part (A):}  Suppose that there are corresponding pairs of points $x_i,x_{i+1}$ and $y_i,y_{i+1}$ such that the dividing sets in both disks connects these points.  We refer to these as {\it matched bypasses}.  Subdivide the edges between $x_{i-1},x_i$ and between $x_{i+1},x_{i+2}$.  Add an edge to $g_1(G)$ through $F_1$ connecting these new vertices and disjoint from $\Gamma_1$.  This separates $F_1$ into two faces, $F'_1$ and $D_1$, where $D_1$ contains the bypass.  Consequently, $\tb(F'_1) = -n + 1$, $\tb(D_1) = -1$ and $\rot(\del F'_1) = \rot(\del F_1) - (-1)^i$.  Add a corresponding edge to $g_2(G)$ through $F_2$.  The unknots $\del F'_1,\del F'_2$ have the same classical invariants and thus are Legendrian isotopic.  By induction, there exists sequences of bypasses attached within $F'_1,F'_2$ that equate the dividing sets in these subfaces.  After attaching these bypasses, the dividing sets in $F_1,F_2$ are now isotopic rel boundary as well.

{\bf Part (B):}  Suppose that the points $x_{i+2},x_{i+3}$ are connected by a dividing arc and that $y_{i}$ and $y_{i+1}$ are connected by a dividing arc, as in Figure \ref{subfig:adjacent-options}.  We refer to these as {\it adjacent bypasses}.  Then we can find a bypass along an arc $a$ in $ F_1$ from $x_i$ to $x_{i+2}$ and a bypass along an arc $b$ in $F_2$ from $y_{i+1}$ to $y_{i+3}$.  Each bypass must be attached either behind or in front and the potential effect is summarized in Figure \ref{subfig:bypass-options}.  

In all four cases, we can find a pair of matched bypasses and reduce to the previous step:  If the bypass along $a$ is attached in front, then attach this bypass and do not attach the bypass along $b$.  As a result, the dividing sets connect $x_i$ to $x_{i+1}$ and $y_i$ to $y_{i+1}$ and so we have found matched bypasses.  We can similarly find matched bypasses if the bypass along $b$ is attached behind.  Finally, suppose $a$ is attached behind and $b$ is attached in front.  After attaching both bypasses, the dividing sets connect $x_{i+1}$ to $x_{i+2}$ and $y_{i+1}$ to $y_{i+2}$ and we have found matched bypasses.

{\bf Part (C):}  We now show that after attaching a sequence of bypasses, we can ensure that one of the above two cases occurs.

Start by attaching bypasses along arcs in $F_1$ from $x_1$ to $x_3$ and in $F_2$ from $y_1$ to $y_3$ as shown in Figure \ref{subfig:bypass-options}.  
If they are both attached on the same side, then we have matching bypasses as in part (A).  The other possibility is that they are attached on opposite sides and without loss of generality we assume that the bypass on $F_1$ is attached behind and the bypass on $F_2$ is attached in front.

Now, attach a bypass along an arc $a_2$ in $F_1$ from $x_3$ to $x_5$ and along $b_2$ in $F_2$ from $y_4$ to $y_6$.  See Figure \ref{subfig:induction-bypass-options}.  If the bypass in $F_1$ is attached in front, the new dividing arc connects $x_4$ to $x_5$ and forms a bypass adjacent to the bypass formed by the arc connecting $y_2$ to $y_3$.  If the bypass in $F_1$ is attached behind and the bypass in $F_2$ is attached in front, then there are adjacent bypasses formed by arcs from $x_3$ to $x_4$ and from $y_5$ to $y_6$.  The only remaining option is that both are attached behind.

In this case, repeat the previous step by attaching bypasses along the arcs from $x_5$ to $x_7$ and from $y_6$ to $y_8$.  If this does not result in matched or adjacent bypasses, repeat again.  Continue attaching bypasses along the arcs from $x_{2i-1}$ to $x_{2i+1}$ and from $y_{2i}$ to $y_{2i+2}$.
Suppose that after $i$ pairs of bypass attachments, we have not found corresponding destabilizations.  
Then there are $i$ dividing arcs in $F_1$ connecting $x_{2j-1}$ to $x_{2j}$ and $i$ dividing arcs in $F_2$ connecting $y_{2j}$ to $y_{2j+1}$ for $j = 1,\dots,i$.  
This implies that $\del F_1$ admits $i$ positive destabilizations and $\del F_2$ admits $i$ negative destabilizations.  
So $\rot(\del F_1) \geq -n + 1 + 2i$ and $\rot(\del F_2) \leq  n - 1 - 2i$.  Moreover, since $\rot(\del F_1) = \rot(\del F_2)$ this implies that $2i \leq n - 1$.  
Thus if $i>\frac{n-1}{2}$ we reach a contradiction.  Since we can repeat the procedure $n-1$ times, at some point we must find matched or adjacent bypasses.
\end{proof}

\subsection{Legendrian graphs in $\cP_k$}
\label{sub:general}

Finally, we use the results of Subsection \ref{sub:base} to prove Theorem \ref{thrm:complete} in general.  

Let $g_1,g_2$ be Legendrian embeddings of $G$ in $\cP_k$ with the same invariants $R_g,\rot_g$.  To prove that $g_1,g_2$ are Legendrian isotopic we use the following strategy.  First, we find convex surfaces $S_1,S_2$ of genus $k$ containing $g_1(G),g_2(G)$ in their characteristic foliations.  Next, we extend $G$ to a graph $Q$ with and extend $g_1,g_2$ to Legendrian embeddings $q_1,q_2: Q \rightarrow (S^3,\xi_{std})$ such that $q_i(Q) \subset S_i$.  There is a subgraph $J \subset Q$ such that $j_1 := q_1|_J$ and $j_2 := q_2|_J$ are Legendrian embeddings in $\cP_0$ and such that $j_1,j_2$ have the same invariants $R_j,\rot_j$.  Thus, by Proposition \ref{prop:complete} there is a Legendrian isotopy $h$ such that $h \circ j_1 = j_2$.  Finally, we show that $h$ extends to a Legendrian isotopy of $q_1,q_2$ and therefore restricts to an isotopy of $g_1,g_2$. 

\subsubsection{The surface $S_g$}

Let $g$ be a Legendrian embedding of $G$ with $p(g)$ edges with $\tw(e,\Sigma_g)>0$.  We will call such edges {\it positive edges}.  Take the sphere $\Sigma_g$ and attach an unknotted 1-handle to $\Sigma_g$ in a neighborhood of each positive edge of $g$.  Let $S_g$ be the resulting genus $p(g)$ surface.  For each positive edge, let $m_e$ and $l_e$ denote the corresponding meridan and longitude of the handle.  We assume that $S_g$ is oriented so that $m_e$ bounds a disk behind $S_g$ and $l_e$ bounds a disk in front.  In addition, we assume that the pair has oriented intersection number $\langle m_e, l_e \rangle = 1$.  See Figure \ref{subfig:handle3Dml}.

Let $k_e = \lceil \tw(e,\Sigma_g) \rceil$.  We can isotope $S_g$ so that $G$ lies on $S_g$ and the edge $e$ crosses the handle and winds $k_e$-times positively around it.  Specifically, there is an orientation on $e$ so that its oriented intersection numbers with the meridian and longitude are $\langle m_e, e \rangle = 1$ and $\langle e, l_e \rangle = k_e$.  Index the points of $e \cap l_e = \{z_1,\dots,z_k\}$ according to their position on $e$.  See Figure \ref{subfig:handle3Dml}.

\begin{figure}[htpb!]
\centering
\begin{subfigure}{.8\textwidth}
	\centering
	\labellist
		\small\hair 2pt
		\pinlabel $e$ at 130 70
		\pinlabel $m_e$ at 250 205
		\pinlabel $l_e$ at 240 70
		\pinlabel $z_3$ at 202 105
		\pinlabel $z_2$ at 240 140
		\pinlabel $z_1$ at 270 120
	\endlabellist
	\includegraphics[width=1\textwidth]{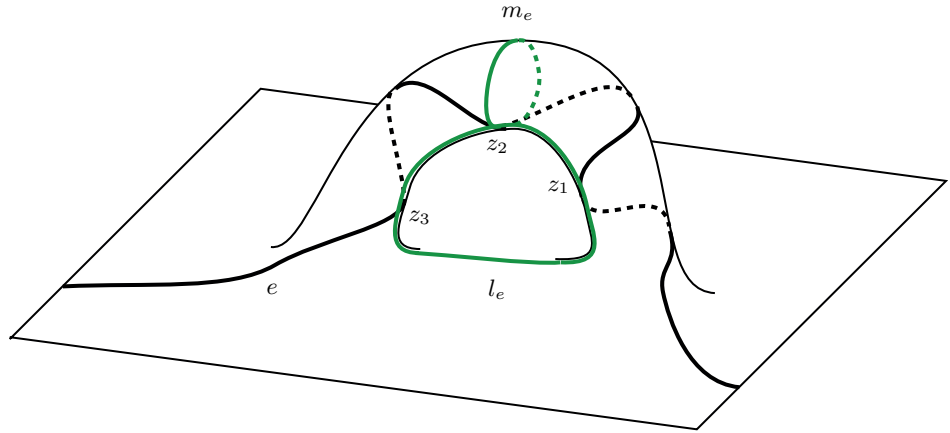}
	\caption{Attach a handle to $\Sigma_g$ with meridian $m_e$ and longitude $l_e$.  Isotope $S_g$ so that $e$ winds $k_e$-times around the handle.}
	\label{subfig:handle3Dml}
\end{subfigure}
\begin{subfigure}{.8\textwidth}
	\centering
	\labellist
		\small\hair 2pt
		\pinlabel $b_1$ at 160 65
		\pinlabel $a_1$ at 130 115
		\pinlabel $x_1$ at 110 65
		\pinlabel $y_1$ at 210 75
		\pinlabel $c_1$ at 240 92
		\pinlabel $c_2$ at 220 135
		\pinlabel $y_2$ at 260 73
		\pinlabel $a_2$ at 270 55
		\pinlabel $x_2$ at 295 65
		\pinlabel $d$ at 188 180
		\pinlabel $b_2$ at 370 75
	\endlabellist
	\includegraphics[width=1\textwidth]{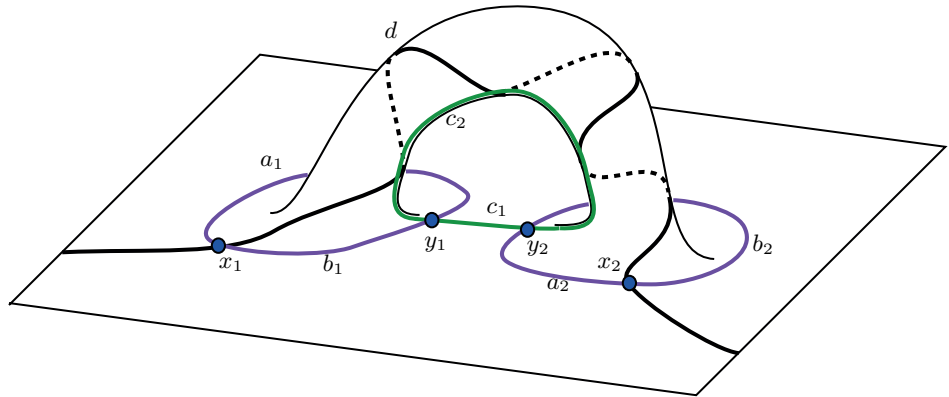}
	\caption{Choose points $x_1,y_1,x_2,y_2$ and arcs $a_1,b_1,a_2,b_2$.}
	\label{subfig:handle3DHgraph}
\end{subfigure}
\begin{subfigure}{.8\textwidth}
	\centering
	\includegraphics[width=1\textwidth]{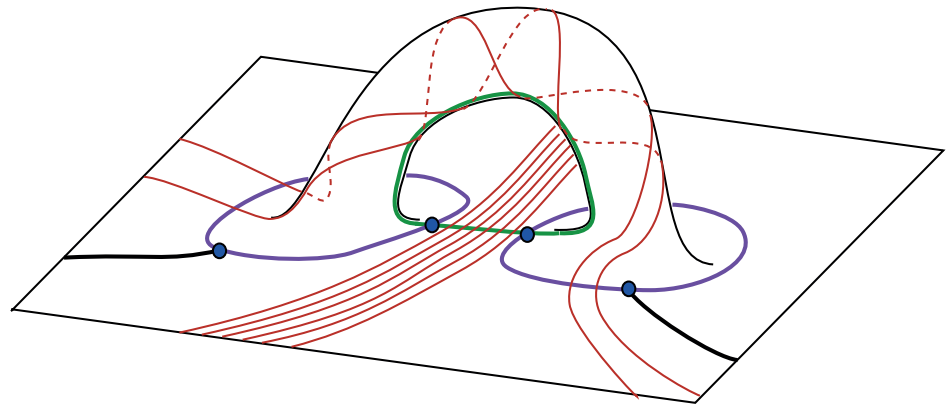}
	\caption{The dividing set near the handle.}
	\label{subfig:handle3Ddividing}
\end{subfigure}
\caption{The handle of the surface $S_g$ near the positive edge $e$.}
\end{figure}

The twisting number of $e$ relative to $S_g$ is now $\tw(e, \Sigma_g) - k_e$ which equals $-\frac{1}{2}$ or $0$.  Consequently, there is a $C^0$-small perturbation of $S_g$ fixing $g(G)$ so that $S_g$ is convex and contains $g(G)$ in its characteristic foliation.

Fix two points $y_1,y_2$ on the longitude $l_e$ such that there is an arc $c_2$ from $y_1$ to $y_2$ in $l_e$ that hits the points $z_1,\dots,z_{k_e}$ in order.  The other arc $c_1$ of $l_e$ is therefore disjoint from the edge $e$.  Index $y_1,y_2$ so that the oriented boundary of this arc is $y_1 - y_2$.  Furthermore, fix two points $x_1,x_2$ on $e$ such that the arc $d$ from $x_1$ to $x_2$ intersects the longitude $k_e$-times.

Finally, choose four arcs $a_1,b_1,a_2,b_2$  such that
\begin{enumerate}
\item the four arcs are mutually disjoint, disjoint from $g(G)$, and disjoint from $l_e$
\item the arcs have oriented boundaries $\del a_i = x_i - y_i$ and $\del b_i = y_i - x_i$
\item the loops $\{a_i \cup b_i\}$ are oriented meridians of the handle isotopic to $m_e$.
\end{enumerate}

See Figure \ref{subfig:handle3DHgraph}.

\begin{lemma}
\label{lemma:handle-intersections}
Let $S_g$ be the genus $k$ convex surface with dividing set $\Gamma$ and containing $g(G)$ in its characteristic foliation.  After possibly attaching some bypasses and isotoping the dividing set, we can assume that
\begin{enumerate}
\item $\# (d \cap \Gamma) = 0$
\item $\# (m_e \cap \Gamma) = 2$
\item $\# (a_i \cap \Gamma) = 2$ and $\#(b_i \cap \Gamma) = 0$
\item $\# (c_2 \cap \Gamma) = 2k_e - 2$
\item $\# (c_1 \cap \Gamma) = 2k_e$
\end{enumerate}
\end{lemma}

See Figure \ref{subfig:handle3Ddividing}.

\begin{proof}
{\bf Part (1):}  If $\tw(e,\Sigma_g)$ is a half-integer, then $\tw(e,S_g) = -\frac{1}{2}$ and so the dividing set $\Gamma$ must intersect $e$ exactly once.  We can isotope this intersection outside of the arc $d$ to lie between $x_2$ and $v_2$, the vertex at the head of the oriented edge $e$.  Otherwise, $\tw(e,S_g) = 0$ and so $\Gamma$ is disjoint from $e$. 

{\bf Part (2):}  The meridian $m_e$ must intersect the dividing set.  Otherwise, the meridian can be Legendrian realized and the compressing disk it bounds is an overtwisted disk, which violates tightness.  After a perturbation, we can assume that the meridian bounds a convex compressing disk $D$ with $\frac{1}{2}\#(m_e \cap \Gamma)$ dividing arcs.  Suppose that the dividing set on $D$ contains a boundary-parallel arc that is isotopic to an arc in $\del D = m_e$ disjoint from the graph $g(G)$.  Unless this arc is the only component of the dividing set of $D$, we can attach a bypass to the back of $S_g$.  The graph $g(G)$ still lies on the surface $S_g$ and this bypass attachment reduces $\#(m_e \cap \Gamma)$ by 2.  Furthermore, such a boundary-parallel arc must exist as long as $\# (m_e \cap \Gamma) \geq 4$, since $g(G)$ intersects $m_e$ in exactly one point.  Thus, we can continue attaching bypasses and reducing $\#(m_e \cap \Gamma)$ until it equals 2.

{\bf Part (3):}  Since $\Gamma$ intersects the meridian $m_e$ twice, we can assume it intersects the meridian $a_i \cup b_i$ exactly twice as well.  By an isotopy, we can assume these intersections lie in $a_i$ and not $b_i$.

{\bf Part (4):} This follows from parts (1) and (3), along with the fact that $\langle d, l_e \rangle = k_e$.

{\bf Part (5):}  Finally, let $K$ be the union of the four arcs $a_1,c_1 ,b_2,d$.  This is a closed loop that is an unknot in $S^3$.  The difference between the surface framing induced by $S_g$ and its nullhomologous framing is $k_e$, so from parts (1) and (2) we can deduce that
\begin{align*}
\tw(c_1,S_g) & = \tw(K,S_g) - \tw(a_1, S_g) - \tw(b_2, S_g) - \tw(d, S_g)\\
& = \tw(K,S_g) + 1 \\
& = \tb(K) + k_e + 1 \\
&\leq k_e
\end{align*}
where the last inequality follows since $\tb(K) \leq -1$.  Using the identity $\tw(c_1,S_g) = -\frac{1}{2} \#(c_1 \cap \Gamma)$, this implies that $c_1$ intersects the dividing set at least $2k_e$-times.

Conversely, choose a convex realization of the compression disk $D$ bounded by $l_e$.  By part (4) and the immediately preceeding argument, we have $\#( l_e \cap \Gamma) \geq 4k_e - 2$.  This implies that $D$ has $ \geq 2k_e - 1$ dividing arcs, with endpoints alternating along $l_e$ with the intersections of $l_e$ with $\Gamma$.    Up to isotopy, we can assume that $2k_e -1$ endpoints lie in $c_2$ and with the remaining endpoints in $c_1$.

We can furthermore conclude that no dividing arc connects two endpoints on $c_1$.  If so, this implies that we can find an arc of attachment along $l_e$ with midpoint at one of the intersection points of $c_1$ and $\Gamma$.  After attaching a bypass in front of $S_g$, we can then find a meridian of the 1-handle disjoint from the new dividing set.  This violates tightness since it can be Legendrian realized and therefore bounds an overtwisted compressing disk.

Now, suppose that $\#(c_1 \cap \Gamma) \geq  2k_e$.  Then by the Imbalance Principle, there must be a boundary-parallel dividing curve on the compressing disk $D$ with endpoints on $c_1$.  This corresponds to a bypass arc of attachment lying completely inside $c_1$.  In particular, the arc of attachment is disjoint from the graph $g(G)$ and hence we can attach this bypass and reduce $\#(c_1 \cap \Gamma)$ by 2, keeping $g(G)$ fixed.  Repeating this argument until $\#(c_1 \cap \Gamma) = 2k_e$ proves part (5).
\end{proof}

The loops $a_1 \cup b_1$ and $a_2 \cup b_2$ bound an annulus $A_e$ in $S_g$ and the arc $d$ is the restriction of $e$ to $A_e$.  Let $d_k$ be an embedded arc from $x_1$ to $x_2$ in $A_e$, homologous to $d + k[m_e]$ and isotoped to have minimal intersection with $\Gamma$.  By a perturbation of $S_g$, we can assume $d_k$ is Legendrian.  Furthermore, isotope $m_e$ to intersect $d_k$ exactly once, orient $d_k$ and $m_e$ so that the intersection point of $m_e \cap d_k$ is positive.  Let $\overline{d}_k$ be a Legendrian realization of the arc obtained by resolving the intersection point $m_e \cap d_k$.  Note that $\overline{d}_k$ is isotopic to $d_{k+1}$ in $A_e$ but not necessarily isotopic rel $\Gamma$.  See Figure \ref{fig:meridian-resolve}.

\begin{lemma}
\label{lemma:d-arc-replace}
Let $d_k$ denote a Legendrian arc in $A_e$ as described above.  
\begin{enumerate}
\item For all $k$, the arcs $d_k$ and $\overline{d}_k$ are Legendrian isotopic,
\item The arcs $b_1 \cup c_1 \cup b_2$ and $d_{-k_e }$ are Legendrian isotopic rel $x_1,x_2$
\item If $k \geq 0$, then $d$ is Legendrian isotopic rel boundary to $d_k$.
\item If $k < 0$, then $d$ is Legendrian isotopic rel boundary to $d_k$ after applying $-k$ positive and $-k$ negative stabilizations to $d$.
\end{enumerate}
In addition $\# \Gamma \cap d_k = 2 |k|$ for all $k \in \ZZ$.
\end{lemma}

\begin{proof}
Let $\alpha = dz - y dx$ be the standard contact form on $\RR^3$.  Let $D$ be a compressing disk bounded by $m_e$.  After possibly a perturbation of $S_g$ and $D$, we can find a contactmorphism between $\nu(D)$ and  $[-\pi,\pi] \times \RR^2 \in \RR^3$ such that: (1) the handle is identified with the surface $A = \{y^2 + z^2 = 1\}$, the meridian $m_e$ is identified with the loop $A \cap \{x = 0\}$, the arc $d_k$ is identified with the line $(x,0,1)$, and $\overline{d}_k$ is identified with the curve $(x,\sin x, - \cos x)$.  Then for $t \in [0,1]$ there is a family $a_t$ of Legendrian curves $a_t = (x,t \sin x, -t \cos x)$.  This gives a Legendrian isotopy between $d_k$ and $\overline{d}_k$ and proves Part (1).

For Part (2), first note that replacing $b_2$ with $a_2$ corresponds to a stabilization of sign $-\sigma(v_1,\Sigma_g)$.  This is true since there is a path from $x_2$ to $v_2$ that crosses $\Gamma$ an even number of times.  To prove the statement, we will show that the arc $d_{-k_e}$ is also destabilization of the arc $b_1 \cup c_1 \cup a_2$ of sign $-\sigma(v_1,\Sigma_g)$.  Attach a compressing disk along a pushoff of $l_e$.  See Figure \ref{subfig:disk-smoothA}.  Recall from Lemma \ref{lemma:handle-intersections} that the dividing set on the compressing disk is $2k_e - 1$ arcs from $c_1$ to $c_2$.  After smoothing the corners, the dividing set consists of $2k_e$ arcs from $d_{-k_e}$ to $b_1 \cup c_1 \cup a_2$ and one arc connecting $c_1$ to itself, which encloses a bigon of sign $-\sigma(v_1,\Sigma_g)$.  See Figure \ref{subfig:disk-smoothB}.  

Part (3) follows from Part (1) since $\overline{d}_k$ is Legendrian isotopic to $d_{k+1}$ for $k \geq 0$.  Part (4) also follows from Part (1) since if $k < 0$ then $\Gamma$ and $\overline{d}_k$ form two trivial bigons, one positive and one negative.  See Figure \ref{fig:meridian-resolve}.  After removing these, which corresponds to a positive destabilization and a negative destabilization, the arc is isotopic rel $\Gamma$ to $d_{k+1}$.
\end{proof}

\begin{figure}[htpb!]
\centering
\labellist
	\small\hair 2pt
	\pinlabel $m$ at 40 120
	\pinlabel $\Gamma$ at 100 70
	\pinlabel $d_{k}$ at 100 180
	\pinlabel $\overline{d}_k$ at 280 180
	\pinlabel $m$ at 420 120
	\pinlabel $\Gamma$ at 450 40
	\pinlabel $d_k$ at 505 150
	\pinlabel $\overline{d}_k$ at 690 150
\endlabellist
\includegraphics[width=.8\textwidth]{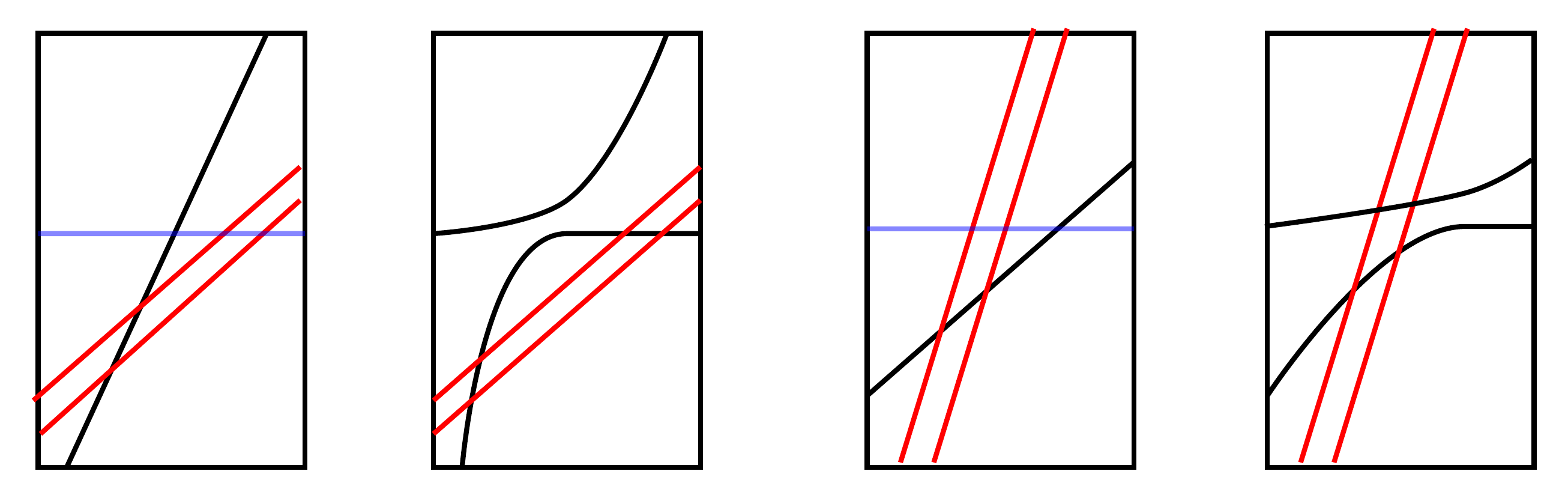}
\caption{Resolve the intersection of $d_k$ with the meridian $m$ yields the arc $\overline{d}_k$.  The case of $k < 0$ is on the left and $k > 0$ is on the right.}
\label{fig:meridian-resolve}
\end{figure}
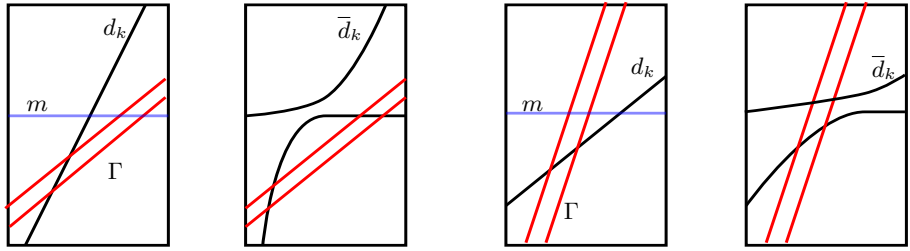

\begin{figure}[htpb!]
\centering
\centering
\begin{subfigure}{.8\textwidth}
	\centering
	\labellist
		\small\hair 2pt
		\pinlabel $b_1$ at 160 82
		\pinlabel $a_2$ at 270 68
		\pinlabel $d_{-k_e}$ at 251 188
		\pinlabel $l_e$ at 293 120
	\endlabellist
	\includegraphics[width=1\textwidth]{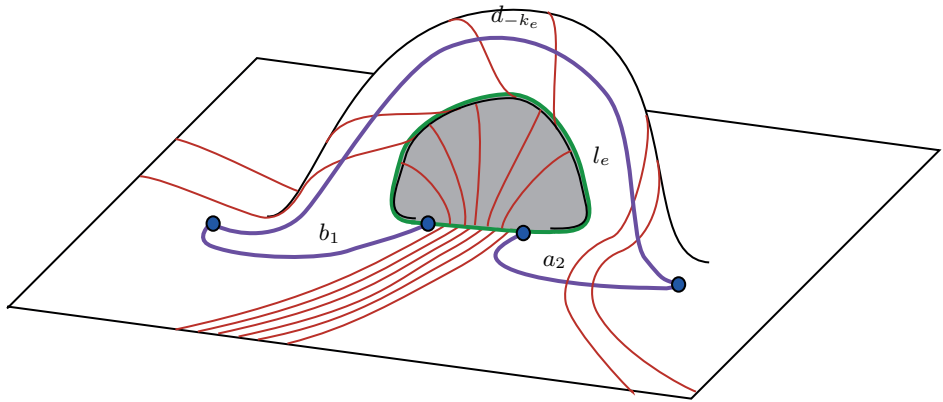}
	\caption{The dividing curves on $S_g$ and a longitudinal compressing disk. }
	\label{subfig:disk-smoothA}
\end{subfigure}
\begin{subfigure}{.8\textwidth}
	\centering
	\labellist
		\small\hair 2pt
		\pinlabel $b_1$ at 160 84
		\pinlabel $a_2$ at 270 73
		\pinlabel $d_{-k_e}$ at 251 192
		
	\endlabellist
	\includegraphics[width=1\textwidth]{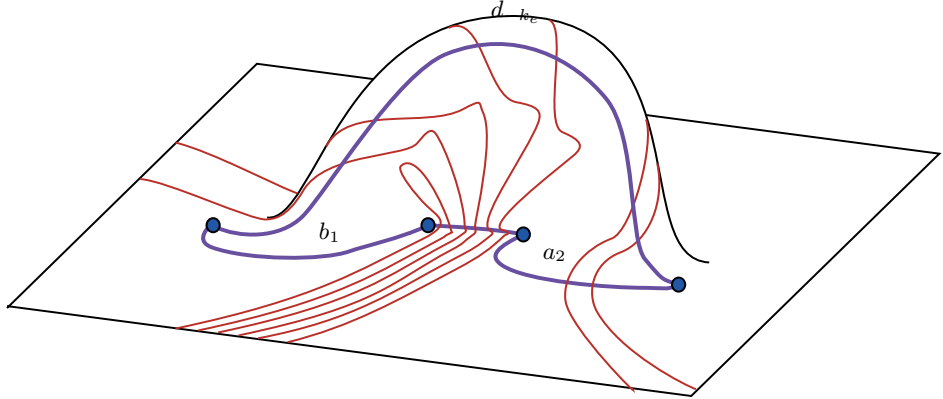}
	\caption{The dividing set after attaching the compressing disk.}
	\label{subfig:disk-smoothB}
\end{subfigure}
\caption{The arc $d_{-k_e}$ is Legendrian isotopic to a destabilization of $b_1 \cup c_1 \cup a_2$.}
\end{figure}

\begin{lemma}
\label{lemma:unknot-longitude}
Fix $k \geq 0$ and arcs $f_1 \in \{a_1,b_1\}$ and $f_2 \in \{a_2,b_2\}$.  Then the cycle $\gamma := d_k \cup f_1 \cup c_1 \cup f_2$ is a Legendrian unknot with $\tb(\gamma) = -1$ and $\rot(\gamma) = 0$.
\end{lemma}

\begin{proof}
Consider the cycle $\gamma = d_k \cup b_1 \cup c_1 \cup b_2$.  The curve $\gamma$ is homologous to $[l_e] + (k + k_e - 1)[m_e]$ and therefore the difference between the surface framing and the nullhomologous framing is $k + k_e - 1$.  In addition, $\# (\gamma \cap \Gamma) = 2k + 2k_e$ by Lemmas \ref{lemma:handle-intersections} and \ref{lemma:d-arc-replace}.  Thus, $\tb(\gamma) = -\frac{1}{2}(2k + 2k_e) + k + k_e - 1 = -1$ and this further implies that $\rot(\gamma) = 0$.  

Replacing $b_i$ with $a_i$ increases the surface framing by 1 and increases $\# (\gamma \cap \Gamma)$ by 2, preserving $\tb(\gamma) = -1$ and $\rot(\gamma) = 0$.
\end{proof}

\subsubsection{The Legendrian graphs $q,j$}
\label{subsub:qj}

Now, using the surface $S_g$, we define two new abstract graphs $Q,J$ and Legendrian embeddings $q: Q \rightarrow (S^3,\xi_{std})$ and $j: J \rightarrow (S^3,\xi_{std})$.

First, we define a genus 0 convex surface $\Sigma_j$ by surgering the handles of $S_g$.  For each handle of $S_g$, we pick two Legendrian meridians according to the sign of $\sigma(v_1,\Sigma_g)$ as follows:
\begin{enumerate}
\item If $\sigma(v_1,\Sigma_g)$ is positive, set $D_e = d$ and set $X_i = x_i$ and $m_i = a_i \cup b_i$ for $i = 1,2$.
\item If $\sigma(v_1,\Sigma_g)$ is negative, replace the arc $d$ with $d_1$.  The arc $d_1$ intersects the dividing set $\Gamma$ twice, separating $d_1$ into three segments: two in $S^-_g$ and on in $S^+_g$.  Let $X_1,X_2$ be points on $d_1$ in $S^+_g$ and let $D_e$ be the subarc of $d_1$ connecting $X_1,X_2$.  Now pick two meridians $m_1,m_2$ of the handle that intersect $d_1$ at $X_1,X_2$, respectively and that intersect $\Gamma$ twice each.
\end{enumerate}

In both cases, the meridians can be Legendrian-realized as unknots with $\tb = -1$.  Remove the annulus bounded by $m_1,m_2$ and replace it with two convex compressing disks bounded by $m_1$ and $m_2$.  Let $\Sigma_j$ denote this surgered surface and after a perturbation, we can assume $\Sigma_j$ is convex.

Now, we define a graph $j: J \rightarrow (S^3,\xi_{std})$ with image on $\Sigma_j$.  For each positive edge $e$, define the embedded graph $J_e$ as follows:
\begin{enumerate}
\item include the subset of $g(e)$ lying on $\Sigma_j$ and the arcs $a_1,b_1,a_2,b_2$ and $c_2$. 
\item if $\sigma(v_1)$ is negative, include the arc $r_1$ along $d_1$ from $x_1$ to $X_1$, the arc $r_2$ along $d_1$ from $X_2$ to $x_2$, and the meridians $m_1,m_2$.  In addition, include two edges $r'_1,r''_1$ parallel to $r_1$ and $r'_2,r''_2$ parallel to $r_2$, one on each side in $\Sigma_j$.
\item include two edges $c'_2,c''_2$ parallel to $c_2$, one on each side in $\Sigma_j$.
\item if $\sigma(v_2,\Sigma_g) \neq \sigma(v_1,\Sigma_g)$, include two parallel arcs of the arc in $e$ from $x_2$ to $v_2$, one on each side.
\end{enumerate}

Let $j(J)$ be the Legendrian graph obtained by replacing $e$ with $J_e$ for each positive edge $e$.  See Figure \ref{figure:Je} for the graph $J_e$ with the vertices $v_1,v_2$ have signs $\sigma(v_1) = -1$ and $\sigma(v_2) = +1$.  For other cases, the graph $J_e$ is a subgraph of the graph in Figure \ref{figure:Je}.

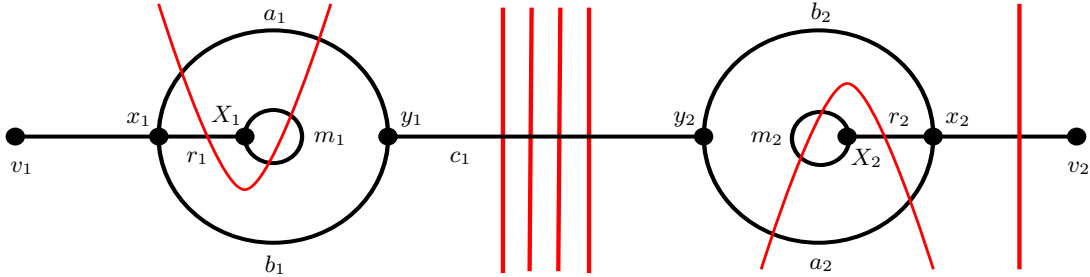
\begin{figure}
\centering
\labellist
	\small\hair 2pt
	\pinlabel $y_1$ at 235 90
	\pinlabel $a_1$ at 165 143
	\pinlabel $b_1$ at 165 15
	\pinlabel $r_1$ at 125 70
	\pinlabel $x_1$ at 95 90
	\pinlabel $X_1$ at 140 92
	\pinlabel $v_1$ at 35 65
	\pinlabel $m_1$ at 193 80
	\pinlabel $c_1$ at 260 70
	\pinlabel $y_2$ at 375 90
	\pinlabel $m_2$ at 417 80
	\pinlabel $X_2$ at 468 70
	\pinlabel $r_2$ at 485 90
	\pinlabel $x_2$ at 515 90
	\pinlabel $v_2$ at 577 65
	\pinlabel $a_2$ at 445 15
	\pinlabel $b_2$ at 445 145
\endlabellist
\includegraphics[width=1\textwidth]{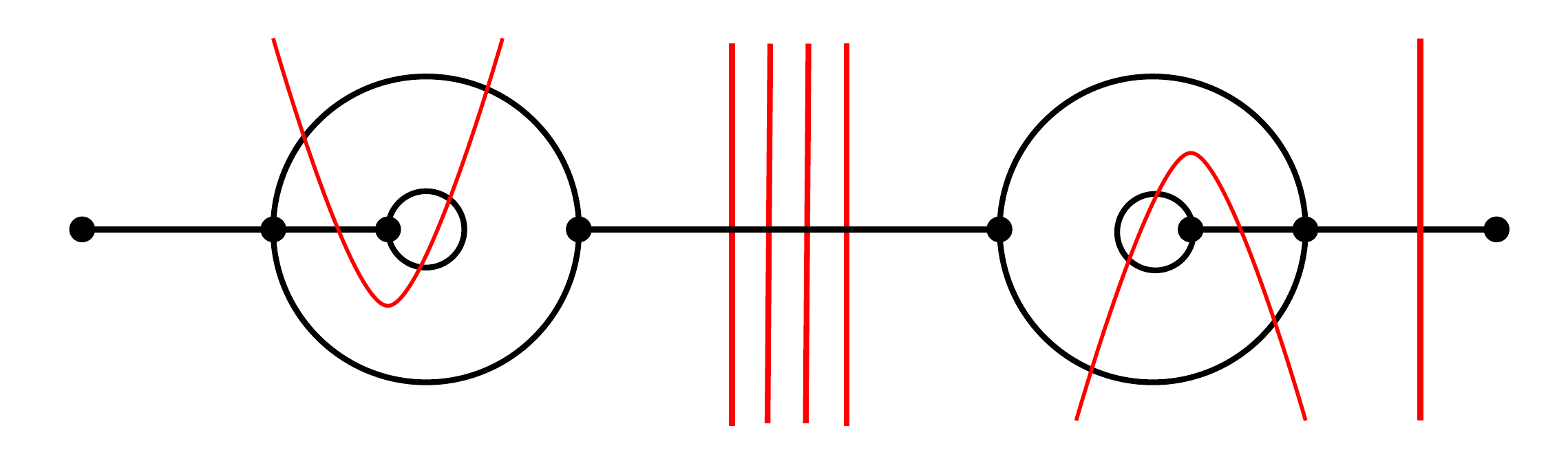}
\caption{The graph $J_e$ when $\sigma(v_1) = -1$ and $\sigma(v_2) = +1$.   For the sake of clarity, the parallel edges are left out.  If $\sigma(v_1) = +1$, remove the edges $r_1,m_1,r_2,m_2$.  If $\sigma(v_2) = \sigma(v_1)$, then $\Gamma$ does not intersect the edge connecting $x_2,v_2$.}
\label{figure:Je}
\end{figure}

\begin{lemma}
The graph $j: J \rightarrow (S^3,\xi_{std})$ can be Legendrian realized on $\Sigma_j$.  In addition, the convex surface $\Sigma_j$ satisfies the conclusions of Corollary \ref{cor:sigma-g}.
\end{lemma}

\begin{proof}
To apply the Legendrian realization principle to $j(J)$, the dividing curve $\Gamma$ must intersect every face on $\Sigma_j$, which is clear from Figure \ref{figure:Je}.  In addition, the extra parallel edges are added so that $J$ has no cut edges or pair of cut edges.  Thus $j$ trivially satisfies the conclusions of Corollary \ref{cor:sigma-g}.
\end{proof}

Let $q: Q \rightarrow (S^3,\xi_{std})$ be the Legendrian graph whose image is the union of the image of $j(J)$ and the arcs $D_e$ for each positive edge $e$.  The image of $q$ contains the images of $g$ and $j$ and we can choose the Legendrian embeddings $g,j$ so that $g,j$ are restrictions of $q$.

\begin{figure}[htpb!]
\centering
\labellist
	\small\hair 2pt
	\pinlabel $B_+$ at 180 80
	\pinlabel $B_-$ at 100 140
	\pinlabel $a_+$ at 145 90
	\pinlabel $a_-$ at 145 140
	\pinlabel $\alpha$ at 120 70
	\pinlabel $\beta$ at 100 200
	\pinlabel $\Gamma$ at 30 50
\endlabellist
\includegraphics[width=.35\textwidth]{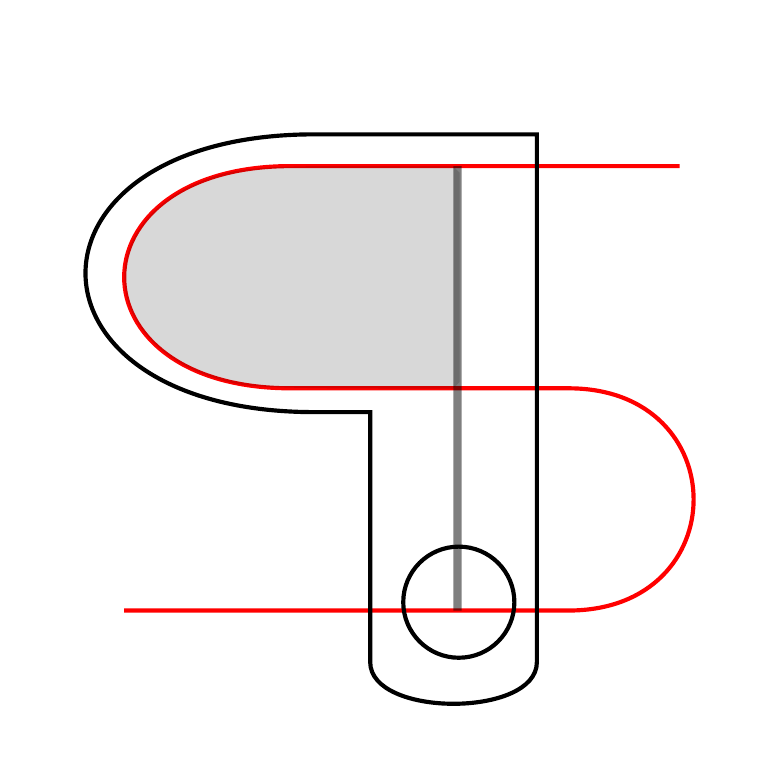}
\caption{The shaded bigon is on the negative side of the dividing curve.  The trivial bypass attached along the arc can be found by pushing the annulus bounded by the pair $\alpha,\beta$ off the sphere and perturbing it to be convex.}
\label{figure:right-2-life}
\end{figure}
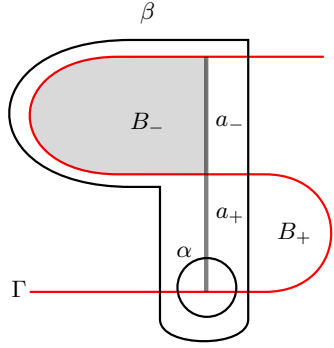

\begin{lemma}
\label{lemma:trivial-bypasses-disjoint}
Suppose that $a$ is a trivial arc of attachment on $\Sigma_j$ in the complement of $j(J)$.  Then we can assume the bypass attached along $a$ is disjoint from $q(Q)$.
\end{lemma}

\begin{proof}
First, the complement of $j(J)$ in $q(Q)$ lies on the front side of $\Sigma_j$, so if the bypass is attached behind it is clearly disjoint from $Q$.

Now, suppose that $a$ is an attachment arc for a trivial bypass in front.  By the following careful analysis of the proof of the Right-to-Life principle, we can ensure that the trivial bypass is disjoint from $q(Q)$.

The arc $a$ and dividing set $\Gamma$ bound two bigons $B_+$ and $B_-$, one positive and one negative.  Let $a_{\pm}$ be the segment of $a$ in the boundary of $B_{\pm}$.  Let $\beta$ be the boundary of a neighborhood of $B_- \cup a$ and let $\alpha$ be the boundary of a neighborhood of the endpoint of $a$ away from $B_-$ and contained in $\beta$.  See Figure \ref{figure:right-2-life}.  The curves $\alpha,\beta$ cobound an annulus $A$.

All of the points of $\del \overline{ \left( q(Q) \smallsetminus j(J) \right)}$ are positive singularities of the characteristic foliation.  Moreover, we can choose $\alpha,\beta$ so that $A$ does not contain any of these singularities.  Now push $\beta$ and $A$ vertically into the front side of $\Sigma_j$.  The annulus $A$ may intersect the surface $S_g$ but we can assume that it is disjoint from the graph $q(Q)$.

Convex realize $A$.  The trivial bypass along $a$ is contained in the annulus $A$, which is disjoint from $q(Q)$.
\end{proof}

\subsubsection{Isotopy}

Let $g_1,g_2$ be Legendrian embeddings of $G$ with the same invariants $R_g,\rot_g$.  Let $q_1,j_1$ and $q_2,j_2$ be the corresponding Legendrian graphs defined in the previous Subsubsection.  We will show that $g_1,g_2$ are Legendrian isotopic.

\begin{lemma}
\label{lemma:same-invts}
Let $g_1,g_2$ be Legendrian embeddings of $G$ with the same invariants $R_g,\rot_g$.  Then
\begin{enumerate}
\item the abstract graphs $Q_1,Q_2$ and $J_1,J_2$ are identical,
\item the graphs $j_1,j_2$ have the same invariants $R_j,\rot_j$, 
\item the graphs $q_1,q_2$ have the same invariants $R_q,\rot_q$, and
\item there exists sequences of bypasses on $\Sigma_{j_1}$ and $\Sigma_{j_2}$ such that, after attaching the bypasses, the dividing sets on $\Sigma_{j_1},\Sigma_{j_2}$ are isotopic rel $j_1(J),j_2(J)$.
\end{enumerate}
\end{lemma}

\begin{proof}
The embeddings $g_1,g_2$ have the same Legendrian ribbon, so they have the same edge invariant $\tw(e,\Sigma_g)$, the same set of positive edges $p(g)$, and each vertex has the same sign $\sigma(v,\Sigma_g)$ for both embeddings.  For each positive edge $e$ connecting a pair of vertices $v_1,v_2$, the graph $J_e$ is determined by $\tw(e,\Sigma_g), \sigma(v_1,\Sigma_g)$ and $\sigma(v_2,\Sigma_g)$.  This proves Part (1).

The Legendrian ribbon $R_j$ can be obtained from a tubular neighborhood of $j(J)$ in $\Sigma_j$ by adding $\tw(e,\Sigma_j)$ twists along each edge $e$.  Since the image of $j(J_e)$ in $\Sigma_j$ is determined by $\tw(e,\Sigma_g), \sigma(v_1,\Sigma_g)$ as well, the graphs $j_1,j_2$ have the same Legendrian ribbon.  We can also obtain $R_q$ as the union of $R_g$ and $R_j$ and so $q_1,q_2$ have the same Legendrian ribbon.

For each cycle $\gamma$ of $J$ and each positive edge $e$, either (1) $\gamma$ is contained in $J_e$, (2) $\gamma$ is disjoint from $J_e$, or (3) $\gamma$ contains a path $p_e$ in $J_e$ from $v_1$ to $v_2$.  In the first case, it is clear from Figure \ref{figure:Je} that either $\rot_j(\gamma) = 0$, or if $\gamma$ is made of two edges between $y_1$ and $y_2$, then $\rot(\gamma)$ depends only on the signs of the vertices $y_1$ and $y_2$.  

Now suppose that $\gamma$ contains paths $\{p_{e_i}\}$ in $\{J_{e_i}\}$ for some collection of positive edges $\{e_i\}$.  We can replace $\gamma$ with a cycle $\widehat{\gamma}$ in $G$ by replacing each path $p_{e_i}$ with the edge $e_i$.  If $p_{e_i} = b_1 \cup c_2 \cup b_2$, then Lemma \ref{lemma:d-arc-replace} implies that $p_{e_i}$ is obtained from $e_i$ by $k_{e_i}$ positive and negative stabilizations.  Replacing the arc $b_1$ with $a_1$ or $b_2$ with $a_2$ is equivalent to a stabilization of sign $-\sigma(v_1,\Sigma_g)$.  Therefore, it follows that the classical invariants of $\widehat{\gamma}$ determine the invariants of $\gamma$.

The same argument holds for any cycle in $Q$.  Thus, we have proved Parts (2) and (3).  Finally, Part (4) follows from Part (2) and Proposition \ref{prop:complete}.
\end{proof}

From the classification of Legendrian unknots, we can deduce that

\begin{lemma}
\label{lemma:isotopic-arcs}
Let $(B^3,\xi_{std})$ be a tight contact 3-ball with convex boundary $S$. Let $a$ be an arc in the characteristic foliation of $S$ connecting two singularities $e_1,e_2$.  Furthermore, let $b_1,b_2: [0,1] \rightarrow B^3$ are Legendrian arcs such that
\begin{enumerate}
\item $b_i(0) = e_1$ and $b_i(1) = e_2$
\item $b_i(t)$ lies in the interior of $B^3$ for $t \in (0,1)$
\item after smoothing, the Legendrian unknots $K_1 = a \cup b_1$ and $K_2 = a \cup b_2$ have $\tb(K_1) = \tb(K_2) = -1$ (and thus $K_1,K_2$ are Legendrian isotopic in $B^3$).
\end{enumerate} 
Then the arcs $b_1,b_2$ are Legendrian isotopic rel boundary.
\end{lemma}

We can now prove the main theorem.

\begin{customthm}{\ref{thrm:complete}}
Let $G$ be an abstract planar graph.  The pair $(R_g,\rot_g)$ is a complete set of invariants of topologically trivial Legendrian embeddings $g: G \rightarrow S^3$.
\end{customthm}

\begin{proof}
Let $g_1,g_2$ be two planar graphs with the same invariants $R_g,\rot_g$.  Let $q_1,j_1$ and $q_2,j_2$ be the corresponding Legendrian graphs defined in Subsubsection \ref{subsub:qj}.  By Lemma \ref{lemma:same-invts}, the graphs $j_1,j_2$ have the same invariants $R_j,\rot_j$ and there are sequences of bypasses on $\Sigma_{j_1},\Sigma_{j_2}$ equating the dividing curves.  By Lemma \ref{lemma:trivial-bypasses-disjoint}, these bypasses can be attached in the complement of $q_1(Q)$ and $q_2(Q)$.  Thus, by Proposition \ref{prop:dividing-equals-isotopy} there is a Legendrian isotopy $h$ such that $h \circ j_1 = j_2$.  Finally, for each positive edge $e$ we can isotope $g_1(D_e)$ and $g_2(D_e)$ to lie in a 3-ball disjoint from $g_1(G \smallsetminus e)$ and $g_2(G \smallsetminus e)$.  Then by combining Lemma \ref{lemma:isotopic-arcs} with Lemma \ref{lemma:unknot-longitude}, we can assume that $h$ extends to an isotopy sending $g_1(D_e)$ to $g_2(D_e)$.  Repeating for each positive edge, we find that $h$ extends to an isotopy between $q_1,q_2$ and thus by restriction, $h \circ g_1 = g_2$.
\end{proof}





\section{Legendrian simplicity}
\label{sec:simple}

In this final section, we use Theorem \ref{thrm:complete} to prove Theorem \ref{thrm:simple}.  Recall that $\Delta_2$ denotes graph with 3 vertices and 2 edges connecting each pair of vertices, and $K_4$ is the complete graph on four vertices.

\begin{customthm}{\ref{thrm:simple}}
Let $G$ be an abstract planar graph and $g: G \rightarrow S^3$ its unique topologically trivial embedding.
\begin{enumerate}
\item If $G$ contains $K_4$ or $\Delta_2$ as a minor, then the pair $(\overline{R}_g,\rot_g)$ is a complete set of invariants.
\item If $G$ is 3-connected, the pair $(\tb_g,\rot_g)$ is a complete set of invariants.
\end{enumerate}
In both cases, $G$ is Legendrian simple.

\end{customthm}

\begin{proof}
First, suppose that $G$ contains either $K_4$ or $\Delta_2$ as a minor.   
Let $g$ be a Legendrian embedding and let $\Sigma_g$ be a sphere containing $g(G)$.  
Pick 3 vertices $v_1,v_2,v_3 \subset G$ that would survive to be distinct vertices in $K_4$ or $\Delta_2$ after contracting edges or deleting vertices.  
By the Pigeonhole Principle, we can assume that, up to reindexing, the signs at $v_1,v_2$ satisfy $\sigma(v_1,\Sigma_g) = \sigma(v_2,\Sigma_g)$.  We can choose 3 oriented paths $p_1,p_2,p_3$ from $v_1$ to $v_2$ and obtain a  corresponding set $C$ of 3 cycles $C_1 = p_1 - p_2; C_2 = p_2 - p_3; C_3 = p_3 - p_1$ as in Subsection \ref{sub:rotation}.  
Since $\sigma(v_1,\Sigma_g) = \sigma(v_2,\Sigma_g)$, the cyclic ordering of $p_1,p_2,p_3$ at the vertices must be opposite one another.  
So Lemma \ref{lemma:total-rot-theta} implies that $\rot_g(C) = 1$.

Now suppose that $h$ is a Legendrian embedding of $G$ with $\rot_h = \rot_g$ and $R_h = -R_g$, i.e. the surface $\overline{R}_g$ with the oppose orientation as $R_g$.  Let $p_1,p_2,p_3$ be the same oriented paths.  However, since $\sigma(v_1,\Sigma_h) = - \sigma(v_1,\Sigma_g)$ and $\sigma(v_2,\Sigma_h) = -\sigma(v_2,\Sigma_g)$, the corresponding fundamental set of cycles $C'$ consists of $C'_1 = p_1 - p_3; C'_2 = p_3 - p_2; C'_3 =p_2 - p_1$.  In other words, $C'$ consists of the same cycles as $C$, but with opposite orientations.  Consequently, $\rot_h(C) = -\rot_h(C') = -1$ and therefore $\rot_h \neq \rot_g$, which is a contradiction.  Thus, if $h$ has contact framing $\overline{R}_h = \overline{R}_g$ and $\rot_h = \rot_g$ then $R_h = R_g$ and by Theorem \ref{thrm:complete}, $h$ is Legendrian isotopic to $g$.

Secondly, suppose that $G$ is 3-connected and let $g$ be a Legendrian embedding of $G$.  Let $\Sigma$ be a smooth 2-sphere containing $g(G)$.  By Whitney's Theorem, the embedding $g: G \rightarrow \Sigma$ is unique up to homeomorphism of $S^2$.  We can obtain the contact framing $\overline{R}_g$ from a tubular neighborhood $R$ of $g(G)$ in $\Sigma$ as follows.  Let $e$ be an edge of $G$, oriented from $v_1$ to $v_2$.  Since $G$ is 3-connected, we can choose two paths $p_1,p_2$ from $v_1$ to $v_2$ that are disjoint from $e$ and vertex-independent from each other.  Set $C_1 = e - p_1; C_2 = p_2 - e;$ and $C_3 = p_1- p_2$.  Then
\begin{align*}
\tw(e,\Sigma_g) &= \frac{1}{2} \left(\tw(e,\Sigma_g) + \tw(p_1,\Sigma_g) + \tw(p_2,\Sigma_g) + \tw(e,\Sigma_g) - \tw(p_1,\Sigma_g) - \tw(p_2,\Sigma_g) \right) \\
& = \frac{1}{2} (\tb(C_1) + \tb(C_2) - \tb(C_3))
\end{align*}
and so the cycle invariant determines $\tw(e,\Sigma_g)$.  Now add $\tw(e,\Sigma_g)$ twists to $R$.  Repeat for each edge to obtain $R_g$. Thus, the cycle invariant $\tb_g$ determines the contact framing $\overline{R}_g$.

To finish the proof, note that if $G$ is 3-connected then it must contain either $K_4$ or $\Delta_2$ as a minor.  Consequently, by Part (1) the rotation invariant $\rot_g$ determines the orientation on $R_g$.

\end{proof}


\bibliographystyle{alpha}
\nocite{*}
\bibliography{References}

\begin{thebibliography}{HKM07}

\bibitem[Ben83]{Bennequin}
Daniel Bennequin.
\newblock Entrelacements et \'equations de {P}faff.
\newblock In {\em Third {S}chnepfenried geometry conference, {V}ol. 1
  ({S}chnepfenried, 1982)}, volume 107 of {\em Ast\'erisque}, pages 87--161.
  Soc. Math. France, Paris, 1983.

\bibitem[BW00]{MR1847313}
Joan~S. Birman and Nancy~C. Wrinkle.
\newblock On transversally simple knots.
\newblock {\em J. Differential Geom.}, 55(2):325--354, 2000.

\bibitem[Che02]{Chekanov}
Yuri Chekanov.
\newblock Differential algebra of {L}egendrian links.
\newblock {\em Invent. Math.}, 150(3):441--483, 2002.

\bibitem[DG07]{Ding-Geiges}
Fan Ding and Hansj{\"o}rg Geiges.
\newblock Legendrian knots and links classified by classical invariants.
\newblock {\em Commun. Contemp. Math.}, 9(2):135--162, 2007.

\bibitem[EF09]{Eliashberg_Fraser}
Yakov Eliashberg and Maia Fraser.
\newblock Topologically trivial {L}egendrian knots.
\newblock {\em J. Symplectic Geom.}, 7(2):77--127, 2009.

\bibitem[EH01]{Etnyre_Honda}
John~B. Etnyre and Ko~Honda.
\newblock Knots and contact geometry. {I}. {T}orus knots and the figure eight
  knot.
\newblock {\em J. Symplectic Geom.}, 1(1):63--120, 2001.

\bibitem[Eli92]{Eliashberg92}
Yakov Eliashberg.
\newblock Contact {$3$}-manifolds twenty years since {J}. {M}artinet's work.
\newblock {\em Ann. Inst. Fourier (Grenoble)}, 42(1-2):165--192, 1992.

\bibitem[Eli93]{Eliashberg93}
Yakov Eliashberg.
\newblock Legendrian and transversal knots in tight contact {$3$}-manifolds.
\newblock In {\em Topological methods in modern mathematics ({S}tony {B}rook,
  {NY}, 1991)}, pages 171--193. Publish or Perish, Houston, TX, 1993.

\bibitem[ENV13]{MR3085098}
John~B. Etnyre, Lenhard~L. Ng, and Vera V{\'e}rtesi.
\newblock Legendrian and transverse twist knots.
\newblock {\em J. Eur. Math. Soc. (JEMS)}, 15(3):969--995, 2013.

\bibitem[Etn99]{Etnyre}
John~B. Etnyre.
\newblock Transversal torus knots.
\newblock {\em Geom. Topol.}, 3:253--268 (electronic), 1999.

\bibitem[FT97]{Fuchs-Tabachnikov}
Dmitry Fuchs and Serge Tabachnikov.
\newblock Invariants of {L}egendrian and transverse knots in the standard
  contact space.
\newblock {\em Topology}, 36(5):1025--1053, 1997.

\bibitem[Gir91]{Giroux91}
Emmanuel Giroux.
\newblock Convexit\'e en topologie de contact.
\newblock {\em Comment. Math. Helv.}, 66(4):637--677, 1991.

\bibitem[Gir93]{MR1246390}
Emmanuel Giroux.
\newblock Topologie de contact en dimension {$3$} (autour des travaux de
  {Y}akov {E}liashberg).
\newblock {\em Ast\'erisque}, (216):Exp.\ No.\ 760, 3, 7--33, 1993.
\newblock S{\'e}minaire Bourbaki, Vol. 1992/93.

\bibitem[Gir00]{MR1779622}
Emmanuel Giroux.
\newblock Structures de contact en dimension trois et bifurcations des
  feuilletages de surfaces.
\newblock {\em Invent. Math.}, 141(3):615--689, 2000.

\bibitem[HKM07]{MR2318562}
Ko~Honda, William~H. Kazez, and Gordana Mati{\'c}.
\newblock Right-veering diffeomorphisms of compact surfaces with boundary.
\newblock {\em Invent. Math.}, 169(2):427--449, 2007.

\bibitem[Hon00a]{Honda1}
Ko~Honda.
\newblock On the classification of tight contact structures. {I}.
\newblock {\em Geom. Topol.}, 4:309--368, 2000.

\bibitem[Hon00b]{Honda2}
Ko~Honda.
\newblock On the classification of tight contact structures. {II}.
\newblock {\em J. Differential Geom.}, 55(1):83--143, 2000.

\bibitem[Kan98]{Kanda}
Yutaka Kanda.
\newblock On the {T}hurston-{B}ennequin invariant of {L}egendrian knots and
  nonexactness of {B}ennequin's inequality.
\newblock {\em Invent. Math.}, 133(2):227--242, 1998.

\bibitem[Mas69]{Mason}
W.~K. Mason.
\newblock Homeomorphic continuous curves in {$2$}-space are isotopic in
  {$3$}-space.
\newblock {\em Trans. Amer. Math. Soc.}, 142:269--290, 1969.

\bibitem[Ng03]{Ng}
Lenhard~L. Ng.
\newblock Computable {L}egendrian invariants.
\newblock {\em Topology}, 42(1):55--82, 2003.

\bibitem[OP12]{OP-1}
Danielle O'Donnol and Elena Pavelescu.
\newblock On {L}egendrian graphs.
\newblock {\em Algebr. Geom. Topol.}, 12(3):1273--1299, 2012.

\bibitem[OP14]{OP-Theta}
Danielle O'Donnol and Elena Pavelescu.
\newblock Legendrian {$\theta$}-graphs.
\newblock {\em Pacific J. Math.}, 270(1):191--210, 2014.

\bibitem[Whi32]{Whitney}
Hassler Whitney.
\newblock Congruent {G}raphs and the {C}onnectivity of {G}raphs.
\newblock {\em Amer. J. Math.}, 54(1):150--168, 1932.

\end{thebibliography}

\end{document}